\pdfoutput=1
\documentclass[11pt,letterpaper]{article}
\usepackage[margin=1in]{geometry}

\usepackage{amsmath,amssymb}
\usepackage{amsthm}
\usepackage[T1]{fontenc}
\usepackage[utf8]{inputenc}
\usepackage{tgtermes}
\usepackage{newtxtext}
\usepackage{newtxmath}

\usepackage{bm}

\usepackage{mathtools}
\usepackage{microtype}
\usepackage{graphicx}
\usepackage[dvipsnames]{xcolor}
\usepackage{booktabs}
\usepackage{makecell}
\usepackage{enumitem}
\usepackage[font=small,labelfont=bf]{caption}
\usepackage{authblk}

\usepackage[sort&compress]{natbib}
\bibpunct[, ]{[}{]}{,}{n}{}{,}

\usepackage[colorlinks=true,linkcolor=MidnightBlue,citecolor=MidnightBlue,urlcolor=MidnightBlue]{hyperref}
\usepackage{algorithm}
\usepackage{algpseudocode}

\theoremstyle{plain}
\newtheorem{theorem}{Theorem}[section]
\newtheorem{lemma}[theorem]{Lemma}
\newtheorem{proposition}[theorem]{Proposition}
\newtheorem{corollary}[theorem]{Corollary}
\newtheorem{assumption}[theorem]{Assumption}
\theoremstyle{definition}
\newtheorem{remark}[theorem]{Remark}
\newtheorem{definition}[theorem]{Definition}

\makeatletter

\newenvironment{proof}[1]{\par\normalfont\topsep6\p@\@plus6\p@\relax
  \trivlist\item[\hskip\labelsep\itshape #1]\ignorespaces}{\endtrivlist\@endpefalse}
\newcommand{\Halmos}{\ifmmode\quad\square\else
  \leavevmode\unskip\penalty9999\hbox{}\nobreak\hfill\quad\hbox{$\square$}\fi}
\newcommand{\appendixheadings}{%
  \let\old@seccntformat\@seccntformat
  \renewcommand{\@seccntformat}[1]{%
    \ifnum\pdfstrcmp{##1}{section}=0 Appendix~\thesection.\quad\else\old@seccntformat{##1}\fi}}
\makeatother

\newcommand{\E}{\mathbb{E}}

\newcommand{\R}{\mathbb{R}}
\newcommand{\1}{\mathbf{1}}
\newcommand{\spn}{\mathrm{sp}}
\newcommand{\lse}{\mathrm{lse}}
\newcommand{\calF}{\mathcal{F}}
\newcommand{\calG}{\mathcal{G}}
\newcommand{\calH}{\mathcal{H}}
\DeclareMathOperator{\diag}{diag}
\DeclareMathOperator{\Cov}{Cov}

\title{Entropy-Regularized Mean-Field Games\\
with Average-Reward Criterion:\\
Span-Seminorm Contractions and Model-Free Learning}
\author[1]{G\"une\c{s} Akba\c{s}}
\author[1]{Naci Sald\i}
\author[2]{Berkay Anahtarci}
\author[2]{Can Deha Kariksiz}
\affil[1]{Department of Mathematics, Bilkent University, Ankara 06800, T\"urkiye\protect\\
\texttt{gunes.akbas@bilkent.edu.tr}, \texttt{naci.saldi@bilkent.edu.tr}}
\affil[2]{Department of Mathematical Engineering, \"Ozye\u{g}in University, \.Istanbul 34794, T\"urkiye\protect\\
\texttt{berkay.anahtarci@ozyegin.edu.tr}, \texttt{deha.kariksiz@ozyegin.edu.tr}}

\date{}

\hypersetup{
  pdftitle={Entropy-Regularized Mean-Field Games with Average-Reward Criterion: Span-Seminorm Contractions and Model-Free Learning},
  pdfauthor={G\"une\c{s} Akba\c{s}, Naci Sald\i, Berkay Anahtarci, Can Deha Kariksiz},
  pdfkeywords={Mean-field games, average reward, entropy regularization, span seminorm, Q-learning, stochastic approximation}
}

\begin{document}
\maketitle

\begin{abstract}
We study stationary discrete-time mean-field games with finite state and action spaces under the average-reward criterion, with an entropy-regularized objective and population-dependent dynamics and rewards. Under a minorization condition, the soft Q-operator is a span-seminorm contraction for the full transition kernel, so relative value iteration does not require the minorizing measure. When the Lipschitz constant of the kernel is below the minorization mass and the temperature exceeds an explicit threshold, the equilibrium fixed-point operator is a contraction: an equilibrium exists, its population distribution is unique, and a fixed-point iteration converges to it geometrically. We then learn the equilibrium model-free from a simulator, combining soft Q-learning at a fixed population distribution with an empirical transition kernel. We prove a finite-sample bound for soft Q-learning in the span seminorm, with temperature-independent constants, and derive convergence rates for the learned population distribution and policy. A congestion-aware routing game illustrates the results.


\smallskip\noindent\textbf{MSC2020 subject classification:} 91A16, 90C40, 93E35, 68T05
\end{abstract}

\section{Introduction}
\label{sec:intro}

Mean-field games (MFGs) model strategic interaction in large populations of symmetric agents, each of whom is affected by the others only through the distribution of their states. The theory was introduced independently by \citet{huang2006} and by \citet{lasry2007}. In place of the $N$-player game, whose analysis becomes intractable as $N$ grows, one studies a single representative agent facing a given population distribution: the agent solves an optimal control problem for that distribution, and an equilibrium is a distribution that is reproduced when every agent uses the resulting optimal policy. For the models and criteria treated by \citet{huang2006} and \citet{saldi2018}, equilibria of this limiting problem yield approximate Nash equilibria of the $N$-player game when $N$ is large; we do not study the finite-population game here. Applications in operations research and economics range from industry dynamics \citep{WeBeVa08,AdJoWe15} to congestion and routing \citep{CaEtAl22}; see \citet{BeFrYa13} for an introduction and \citet{CaDe18a,CaDe18b} for a comprehensive treatment.

We study stationary discrete-time MFGs with finite state and action spaces under the long-run average-reward criterion and an entropy-regularized objective. Both the transition kernel and the one-stage reward depend on the population distribution. For a fixed population distribution, the representative agent chooses a stationary Markov policy to maximize the long-run average reward together with a temperature-weighted Shannon entropy term. A stationary mean-field equilibrium consists of a policy and a population distribution such that the policy is optimal for the given population distribution, while the population distribution is invariant under the state process induced by that policy. We show that, under a minorization condition and suitable regularity of the transition kernel, and for sufficiently high temperature, the equilibrium population distribution is the unique fixed point of a contraction and can therefore be computed by a geometrically convergent iteration. We then make this iteration model-free, learning the equilibrium from simulated transitions with finite-sample guarantees; the precise conditions and results are stated in Section~\ref{sec:contributions}. Our approach rests on two ingredients: the span-seminorm theory of the average-reward optimality equation under a minorization condition, and the smoothness of the softmax policy produced by entropy regularization.

\subsection{Background and related work}
\label{sec:related}

\paragraph{Discrete-time and ergodic mean-field games.}
Discrete-time games with a continuum of anonymous players go back to the anonymous sequential games of \citet{JoRo88}; related equilibrium notions in operations research and economics include the oblivious equilibrium of \citet{WeBeVa08} and the stationary equilibrium of \citet{AdJoWe15}. \citet{GoMoSo10} established the existence of equilibria for discrete-time MFGs with finite state spaces. For discounted discrete-time MFGs with Polish state and action spaces, \citet{saldi2018} proved the existence of a mean-field equilibrium and showed that the equilibrium policy constitutes an approximate Markov--Nash equilibrium of the corresponding $N$-player game when $N$ is sufficiently large. In continuous time, ergodic MFGs are studied, for instance, by \citet{CaLaLiPo12}, \citet{BaPr14} and \citet{ArBiCa17}. In discrete time, existence of equilibria under the average-reward criterion was established by \citet{WiAl15} for stationary anonymous sequential games with finite state and action spaces, by \citet{Bis15} on $\sigma$-compact Polish state spaces, by \citet{Wie19} on compact spaces under a minorization condition, and by \citet{Sal20} on Polish spaces. These works establish the existence of equilibria and, in some cases, relate them to the $N$-player game, but they do not provide algorithms with convergence guarantees for computing an equilibrium. Such algorithms were developed for average-cost MFGs by \citet{AnKaSa20,anahtarci2019learning}.

\paragraph{Uniqueness, computation and regularization.}
Mean-field equilibria need not be unique. Uniqueness is classically obtained under the monotonicity condition of \citet{lasry2007}, which also underlies the convergence of fictitious play and of online mirror descent in MFGs \citep{ElPeLaGePi20,PeEtAl20,PeEtAl22}. A second approach is to show that the mean-field operator, which maps a population distribution to its one-step push-forward under the corresponding best response, is a contraction; this yields uniqueness together with a convergent algorithm. With finitely many actions, however, the best response requires choosing an action with the highest value, and there may be multiple optimal actions in case of a tie. \citet{CuKo21} showed for finite-horizon finite MFGs that, for any deterministic selection of best responses, the operator mapping a mean-field flow to the flow it generates from a fixed initial distribution is either constant or not even Lipschitz continuous, and \citet{GuHuXuZh19,GuHuXuZh23} argued and illustrated numerically that naively combining reinforcement learning with fixed-point iteration in classical MFGs yields unstable algorithms. The contraction property has been recovered in two ways. For continuous action sets, \citet{anahtarci2019learning} assumed that the one-stage cost plus the expected continuation value is strongly convex in the action, so that the optimal action depends Lipschitz continuously on the population distribution; their analysis covers both the discounted and the average-cost criteria. Alternatively, the best response can be smoothed, either by applying a softmax (Boltzmann) policy to the unregularized action values \citep{GuHuXuZh19,GuHuXuZh23,CuKo21,ZaKoBhBa23} or by adding an entropy-type regularizer to the objective; we discuss the latter approach next.

\paragraph{Entropy-regularized mean-field games.}
Adding a temperature-weighted entropy term to the agent's objective replaces the best response with a softmax of the regularized action values; this policy is unique and depends Lipschitz continuously on those values. The resulting regularized equilibrium, in which agents randomize over their actions, is an approximate equilibrium of the unregularized game, with an error that vanishes as the temperature tends to zero \citep{CuKo21,XiYaWaMi21}.
The regularized mean-field operator is a contraction when the regularization is sufficiently strong, in which case the regularized equilibrium is unique and can be computed by fixed-point iteration. Results of this type are available for finite-horizon games with entropy regularization \citep{CuKo21} and, under population-independent dynamics, with Kullback--Leibler regularization towards a reference policy \citep{guan2022}; for discounted stationary games with a regularizer that is strongly convex in the randomized action \citep{AnKaSa23}; and, under refined contraction conditions, for finite-horizon games \citep{AyBaSa25} and for multi-population finite-horizon and stationary games \citep{AyBa26}.
The sufficient conditions in these results require the regularization to be strong relative to the Lipschitz constants of the model \citep{CuKo21,AnKaSa23,YaCaGeHe23} and the contraction constant depends on the discount factor or the horizon; in the finite-horizon case, the regularization strength required by the available condition grows with the horizon length \citep{guan2022}. Learning algorithms with convergence guarantees likewise target regularized equilibria and rely on contraction conditions \citep{XiYaWaMi21,YaCaGeHe23}. \citet{YaGoHe24} note that such guarantees require heavy regularization, which biases the computed equilibrium. They also show that, without structural assumptions such as contractivity or monotonicity, computing approximate equilibria of stationary and finite-horizon MFGs is PPAD-complete.

In continuous time, entropy-regularized MFGs have been studied in linear-quadratic and exploratory settings \citep{GuXuZa22,FiJa22}; for finite-state models, \citet{YaAyBa26} show that a sufficiently large discount rate yields a horizon-independent contraction, which is not available in the undiscounted case. In discrete time, results under the average-reward criterion are more limited. \citet{GoMoSo10} proved uniqueness of stationary solutions by a contraction argument for a sufficiently large entropy penalty, in a model where players choose transition probabilities directly. More recently, \citet{AyBaSa26} approximate average-reward equilibria of regularized games by finite-horizon equilibria using model-based schemes; their conditions allow multiple stationary equilibria, and their uniqueness results concern the terminal-value problems. For the inverse problem, maximum-causal-entropy methods have been used to explain (unregularized) equilibrium behavior under the average-reward criterion \citep{AlSa25,AlSaAnKa26}. %

\paragraph{Learning mean-field equilibria.}
When the model is unknown but a simulator is available, the fixed-point iteration can be implemented by alternating an inner reinforcement-learning step at a frozen population with an outer population update \citep{GuHuXuZh19,GuHuXuZh23,anahtarci2019learning,AnKaSa23}. These works, like the present paper, assume a simulator that can be queried at an arbitrary population distribution. Other schemes for infinite-horizon discounted MFGs include those of \citet{SuMa19}, \citet{AnFoLa22} and \citet{XiYaWaMi21}. Finite-sample guarantees have also been obtained without a mean-field oracle along a single sample path \citep{ZaKoBhBa23}, and for independent learning by finitely many agents with no control over the population distribution \citep{YaCaGeHe23}; see \citet{LaEtAl22} for a survey. Results for the average-reward criterion are scarcer. \citet{FuYaChWa20} prove linear convergence of a model-free actor-critic method for discrete-time linear-quadratic MFGs with ergodic cost. \citet{anahtarci2019learning} combine fitted Q-iteration with empirical population updates and provide an error analysis under the strong convexity assumption described above. \citet{ZeBhKoGa25} propose a single-loop actor-critic algorithm that learns along a single sample path, with finite-time guarantees for unregularized average-reward MFGs satisfying a structural condition that permits multiple equilibria. \citet{OcEtAl25} establish finite-sample bounds for a trust-region policy optimization method for entropy-regularized MFGs, using a discounted formulation as an approximation of the ergodic one and a reset-and-step oracle weaker than a generative model.

\paragraph{Average-reward MDP tools.}

At a fixed population distribution, the representative agent faces an average-reward MDP, whose optimality equation (AROE) can be solved by relative value iteration \citep{white1963,Put94,HeLa96}. Entropy-regularized MDPs, including the average-reward case, are treated by \citet{NeJoGo17} and \citet{GeScPi19}. Under a minorization condition on the transition kernel \citep[Section~3.3]{Her89}, which has also been used for average-reward MFGs \citep{Wie19,Sal20,AnKaSa20,anahtarci2019learning}, two contraction arguments are available. In the first, the operator built on the residual kernel, obtained by subtracting the minorizing measure, is a supremum-norm contraction; this route underlies the model-based analyses of \citet{AnKaSa20,anahtarci2019learning}, and \citet{AlSaAnKa26} apply it to the soft Bellman operator at a fixed population distribution in the inverse problem. In the second, the operator built on the full kernel is a contraction in the span seminorm \citep[Section~5.5]{HeLa96}. The first route requires the minorizing measure, which is unavailable when the model is unknown; for this reason, \citet{anahtarci2019learning} switch to the span seminorm in the error analysis of their learning algorithm. For a single MDP, relative value iteration Q-learning goes back to \citet{AbBeBo01}; finite-sample bounds in the span seminorm were obtained by \citet{ZhZhMa21}, and a general theory of seminorm-contractive stochastic approximation was developed by \citet{ChZhZhHaMa25}.

\subsection{The gap addressed by this paper}
\label{sec:gap}

To our knowledge, the literature reviewed above leaves a gap for discrete-time MFGs with finite state and action spaces, population-dependent dynamics and the average-reward criterion. The available contraction analysis of the mean-field operator under this criterion \citep{anahtarci2019learning} relies on the strong convexity assumption recalled above, which has no counterpart when the action set is finite, and its model-based part is formulated through the residual kernel and therefore requires the minorizing measure. Entropy regularization over finite actions can be viewed as a strongly concave regularization of the randomized action, as in \citet{AnKaSa23}, but the lifted model does not satisfy the hypotheses of \citet{anahtarci2019learning} as stated, since the entropy is neither Lipschitz continuous nor differentiable at the boundary of the simplex. Apart from \citet{GoMoSo10}, where players choose transition probabilities directly, contraction results for regularized MFGs are limited to the discounted and finite-horizon criteria, and learning guarantees for average-reward MFGs cover only unregularized games \citep{FuYaChWa20,anahtarci2019learning,ZeBhKoGa25} or discounted approximations \citep{OcEtAl25}. What is missing is a treatment of the entropy-regularized average-reward MFG in which a span-seminorm contraction of the soft Q-operator, combined with a contraction of the population update, yields uniqueness and computation of the equilibrium as well as convergence of a simulator-based learning scheme, with explicit constants and finite-sample rates. The present paper provides such a treatment. In our analysis, the minorization mass takes over the role played by the discount factor or the horizon in the contraction results above. The algorithms use this mass only through the step sizes and require no knowledge of the minorizing measure.

\subsection{Main results and contributions}
\label{sec:contributions}

We work under two standing assumptions (Section~\ref{sec:mfg_lip}): the reward and the transition kernel are Lipschitz continuous in the state, the action and the population distribution, and the kernel satisfies a minorization condition uniformly in the population distribution. For the learning results we assume a simulator that, for any state, action and population distribution, returns the reward and an independent next-state sample; the algorithms query all state-action pairs synchronously. Our contributions are as follows.

\begin{enumerate}[label=(\arabic*),leftmargin=*]
\item \emph{The soft optimality equation in the span seminorm.} The soft Q-operator, obtained by replacing the hard maximum with the log-sum-exp, is a span-seminorm contraction with the same modulus and the same full kernel as its unregularized counterpart (Theorem~\ref{thm:Lhat_delta_span}). Soft relative value iteration therefore solves the soft optimality equation without the minorizing measure (Corollary~\ref{cor:Lhat_soft_banach}). We also record a Lipschitz bound for the softmax map, with constant equal to the inverse temperature and invariant under additive constants (Lemma~\ref{lem:softmax_linf_l1}); bounds of this type are known \citep{GaPa18}.

\item \emph{Existence, uniqueness and computation of the equilibrium.} We show that the mean-field operator is Lipschitz continuous on the simplex with an explicit constant (Theorem~\ref{thm:F-contraction}). The proof combines three estimates: the soft Q-function is Lipschitz in the population distribution in the span seminorm, with a constant that does not involve the temperature (Lemma~\ref{lem:lipschitz-Q}); the softmax policy is Lipschitz in the population distribution in the $\ell^1$-norm, with a constant proportional to the inverse temperature (Lemma~\ref{lem:lipschitz-policy}); and the population update satisfies a Dobrushin-type bound (Corollary~\ref{cor:dobrushin}). When the Lipschitz constant of the kernel is smaller than the minorization mass and the temperature exceeds an explicit threshold, the operator is a contraction. A mean-field equilibrium then exists, its population distribution is unique, and the fixed-point iteration converges to it geometrically; the equilibrium policy is the softmax policy at that distribution and is unique on its support (Corollary~\ref{cor:MFNE}).

\item \emph{Finite-sample analysis of soft Q-learning.} At a fixed population we learn the soft Q-function by a synchronous soft relative value iteration Q-learning scheme (Algorithm~\ref{alg:softQL}). With explicit step sizes, the mean-square error in the span seminorm decays inversely with the number of iterations; the constants depend logarithmically on the number of state-action pairs and polynomially on the inverse minorization mass, and they do not depend on the temperature (Theorem~\ref{thm:softQL-rate}). The corresponding bound for the learned softmax policy carries an additional factor proportional to the squared inverse temperature. The key step of the proof subtracts from the error a scalar sequence defined through the minorizing measure, which appears in the analysis only; this turns the span-contractive recursion into a stochastic approximation with a supremum-norm contractive drift (Lemma~\ref{lem:reduction}), to which the finite-sample theory of \citet{ChMaShSh20} applies.

\item \emph{Simulator-based learning of the equilibrium.} Replacing the exact soft Q-function by the output of soft Q-learning, and the exact kernel by an empirical kernel built from fresh samples, gives an inexact fixed-point iteration (Algorithm~\ref{alg:mfg-learning}). At every outer iteration, we bound the expected errors of the learned population distribution and policy by the geometric decay of the exact iteration plus the accumulated inner errors, each of the order of the inverse square roots of the inner sample sizes (Theorem~\ref{thm:mfg-learning}). Consequently, the learned population distribution and policy converge in expectation to the equilibrium when the inner sample sizes grow without bound; with constant sample sizes the error bound settles at a floor of that order, and with geometrically growing sample sizes it decreases geometrically in the number of outer iterations (Corollary~\ref{cor:budgets}).
\end{enumerate}

\subsection{Organization}

Section~\ref{sec:classical} collects the model, the average-reward optimality equation for Q-functions and its span-seminorm solution theory. Section~\ref{sec:entropy} develops entropy regularization: it introduces the softmax policy and the
log-sum-exp value, obtains the soft Q-AROE, transfers the
contraction result to the regularized operator, and proves the Lipschitz
continuity of the softmax. Section~\ref{sec:mfg} formulates the entropy-regularized
average-reward mean-field game, proves that the mean-field operator is a
contraction, and deduces existence, uniqueness, and computability of the equilibrium. Section~\ref{sec:learning} contains the model-free learning algorithm and its analysis. Section~\ref{sec:numerics} reports the numerical experiments, and Section~\ref{sec:conclusion} concludes. The proofs of the results in Sections~\ref{sec:classical} and~\ref{sec:entropy}  are collected in Appendices~\ref{app:sec2} and~\ref{app:sec3}, respectively; the proofs for Sections~\ref{sec:mfg} and~\ref{sec:learning} are given in the main text.

\paragraph{Notation.}
For a finite set $E$, $\Delta(E)$ denotes the set of probability distributions on $E$, equipped with the $\ell^1$-norm $\|p - q\|_1 := \sum_{e \in E}|p(e) - q(e)|$. For $h:E\to\mathbb{R}$ we use the sup-norm and the span seminorm
\[
\|h\|_{\infty}:=\max_{e\in E}|h(e)|,
\qquad
\spn(h):=\max_{e\in E}h(e)-\min_{e\in E}h(e).
\]
The span seminorm vanishes exactly on the constant functions, satisfies $\spn(h+c\1)=\spn(h)$ for every $c\in\mathbb{R}$, where $\1$ denotes the constant function equal to one, and is a norm on the quotient space $\mathbb{R}^E/\mathbb{R}\1$; moreover $\spn(h)\le2\|h\|_\infty$ and $\min_{c\in\R}\|h-c\1\|_\infty=\spn(h)/2$.

\section{Average-Reward MDPs and the Span-Seminorm Approach}
\label{sec:classical}

This section introduces the MDP model and collects the classical facts about the average-reward optimality equation and its solution under a minorization condition. The material is standard \citep[Chapter~5]{HeLa96}, and we include proofs (in Appendix~\ref{app:classical}) only for the span-seminorm contraction, since the same argument is repeated in Section~\ref{sec:entropy} for the entropy-regularized MDP problem.

\subsection{Model, policies and notation}
\label{sec:policies}

\begin{definition}
\label{def:mdp}
A \emph{discrete-time MDP} is a tuple $(S,A,\{A(x)\},P,r)$ with a finite \emph{state space} $S$; a finite \emph{action space} $A$; a \emph{controlled transition kernel} $P(\cdot\mid x,a)\in\Delta(S)$
    for each $(x,a) \in S \times A$; and a \emph{one-stage reward} $r:S\times A\to\mathbb{R}$.
\end{definition}

At each time step $t$, the agent observes the current state $X_t$ and selects an action $A_t \in A$; it receives the reward $r(X_t, A_t)$, and the system moves to $X_{t+1}\sim P(\cdot \mid X_t, A_t)$.

A \emph{stationary Markov policy} is a map $\pi:S\to\Delta(A)$. The policy is \emph{deterministic} if every $\pi(\cdot\mid x)$ is a Dirac mass and \emph{stochastic} otherwise, and $\Pi$ denotes the set of all stationary Markov policies. For the criteria considered in this paper, restricting attention to stationary Markov policies entails no loss of optimality \citep{HeLa96,Put94}. Under $\pi\in\Pi$ the state process is a time-homogeneous Markov chain with the policy-averaged kernel and reward
\begin{equation}
\label{eq:induced}
P^\pi(y\mid x)
:=
\sum_{a\in A}\pi(a\mid x)\,P(y\mid x,a),
\qquad
r^\pi(x)
:=
\sum_{a\in A}\pi(a\mid x)\,r(x,a).
\end{equation}

\subsection{The average-reward criterion and its optimality equation}
\label{sec:acoe}

Let $(X_t, A_t)_{t \ge 0}$ be the state--action process generated by $\pi\in\Pi$ from $X_0 = x$, and let $\mathbb{E}_x^\pi$ denote the corresponding expectation. The \emph{long-run average reward} of $\pi$ is
\[
\rho^\pi(x)
:=
\liminf_{N\to\infty}
\frac{1}{N}
\E_x^\pi\!\left[\sum_{t=0}^{N-1}r(X_t,A_t)\right],
\]
and a policy $\pi^{*}$ is \emph{average-reward optimal} if $\rho^{\pi^*}(x)=\rho^*(x):=\sup_{\pi\in\Pi}\rho^\pi(x)$ for every $x\in S$. Under the minorization condition introduced below, which is assumed throughout, the controlled chain is ergodic under every policy, the $\liminf$ is a limit, and $\rho^\pi(x)$ does not depend on $x$.

The optimal average reward and an optimal policy are characterized by the average-reward optimality equation (AROE), which we write directly in terms of Q-functions: a pair $(\rho,Q)\in\mathbb{R}\times\mathbb{R}^{S\times A}$ solves the \emph{Q-AROE} if
\begin{equation}\label{eq:Q_ACOE}
\rho+Q(x,a)
=
r(x,a)+\sum_{y\in S}P(y\mid x,a) \, \max_{b\in A}Q(y,b),
\qquad (x,a)\in S\times A .
\end{equation}
Setting $h(x):=\max_{a\in A}Q(x,a)-\rho$, \eqref{eq:Q_ACOE} is equivalent to the familiar bias form $\rho+h(x)=\max_{a\in A}\{r(x,a)+\sum_yP(y\mid x,a)h(y)\}$ of the AROE, with $Q(x,a)=r(x,a)+\sum_yP(y\mid x,a)h(y)$. If $(\rho,Q)$ solves~\eqref{eq:Q_ACOE}, then $\rho=\rho^*(x)$ for every $x\in S$ and every stationary policy that selects at each $x$ an action in $\arg\max_{a\in A}Q(x,a)$ is average-reward optimal; see \citet[Theorem~5.2.4]{HeLa96} and \citet[Chapter~8]{Put94}. The scalar $\rho$ is therefore unique, whereas $Q$ is unique only up to an additive constant: since $\sum_yP(y\mid x,a)=1$, the pair $(\rho,Q+c\1)$ solves~\eqref{eq:Q_ACOE} whenever $(\rho,Q)$ does. This is why the span seminorm is the natural metric for~\eqref{eq:Q_ACOE}. Working with Q-functions rather than with the bias $h$ is also what makes the theory compatible with learning: the optimal policy $\arg\max_aQ(x,a)$ needs no knowledge of $P$ at decision time, and the sample-based version of the right-hand side of~\eqref{eq:Q_ACOE} requires no maximization inside the expectation.

\subsection{Minorization and the span-seminorm contraction}
\label{sec:span}

The following condition is a standard tool for the geometric ergodicity of Markov chains and for average-reward control \citep[Section~3.3]{Her89} and is natural for finite or compact state spaces, and has been used for average-cost mean-field games with compact state spaces \citep[Assumption~A.3]{Wie19}.

\begin{assumption}[Minorization condition]
\label{def:minor}
The controlled transition kernel satisfies a \emph{minorization condition} if
there exists a sub-probability measure $\nu$ on $S$ with total mass
$
\varepsilon:=\sum_{y\in S}\nu(y)\in(0,1)
$
such that
$$P(y\mid x,a)\ge\nu(y)$$
for all $x\in S$, $a\in A$,
$y\in S$.
The \emph{residual kernel}
$\widehat{P}(y\mid x,a):=P(y\mid x,a)-\nu(y)\ge 0$ then satisfies
$\sum_y\widehat{P}(y\mid x,a)=1-\varepsilon \in (0,1)$.
\end{assumption}

Under Assumption~\ref{def:minor} there are two ways to solve~\eqref{eq:Q_ACOE} by a contraction argument \citep[Section~5.5]{HeLa96}. The operator obtained by replacing $P$ with the residual kernel $\widehat P$ in the right-hand side of~\eqref{eq:Q_ACOE} is a contraction in the sup-norm with modulus $1-\varepsilon$, but it involves the minorizing measure $\nu$ explicitly, both in the operator and in the formula $\rho=\sum_y\nu(y)\max_bQ(y,b)$ for the optimal average reward; this is unavailable when the model is unknown. The operator built on the full kernel $P$ is a contraction only in the span seminorm, but it uses $P$ and $r$ alone. Since our ultimate goal is learning, we follow the second route; the measure $\nu$ enters the proofs but never the algorithms.

\begin{definition}[Q-operator]
\label{def:Q_operators}
For $Q:S\times A\to\mathbb{R}$ and $(x,a)\in S\times A$, we define
\begin{equation}
(LQ)(x,a)
:=r(x,a)+\sum_{y\in S}P(y\mid x,a)\, \max_{b\in A}Q(y,b).
\label{eq:L_op}
\end{equation}
\end{definition}

Note that $L(Q+c\1)=LQ+c\1$ for every $c\in\R$, and that the Q-AROE~\eqref{eq:Q_ACOE} reads $\rho\1+Q=LQ$.

\begin{theorem}[$L$ is a contraction in $\spn(\cdot)$]
\label{thm:Lhat_span}
Under Assumption~\ref{def:minor}, for all $Q,Q':S\times A\to\mathbb{R}$,
\[
\spn\!\left(LQ-LQ'\right)
\le(1-\varepsilon)\,\spn(Q-Q').
\]
\end{theorem}

The proof is given in Appendix~\ref{app:thm:Lhat_span}.

We now explain how a solution of~\eqref{eq:Q_ACOE} is obtained from this theorem. Fix any $Q_0:S\times A \to \mathbb{R}$ and define the iterates
$Q_{n+1}:=LQ_n$. Since $L$ is invariant under additive constants, the sequence $(Q_n)$ may drift by constants, but the successive
differences contract in the span seminorm:
$\spn(Q_{n+1}-Q_n)\le(1-\varepsilon)^n\,\spn(Q_1-Q_0)\to 0$,
so that $Q_{n+1}(x,a)-Q_n(x,a)$ converges to a constant $\rho^{*}$
independently of $(x,a)$ and of $Q_0$. To state this precisely, fix $(x_0,a_0)\in S\times A$ and define
\[
\mathbb{R}^{S\times A}_{(x_0,a_0)}
:=
\{Q:S\times A\to\mathbb{R}: Q(x_0,a_0)=0\},
\]
a complete metric space
under the span seminorm (on which $\spn(\cdot)$ is a norm equivalent to $\|\cdot\|_\infty$), and the normalized operator
$\widetilde{L}Q(x,a):=LQ(x,a)-LQ(x_0,a_0)$.

\begin{corollary}[Solution to Q-AROE via $L$: relative value iteration]
\label{cor:Lhat_banach}
There exists a unique fixed point $Q^*$ of $\widetilde{L}$ in $\mathbb{R}^{S\times A}_{(x_0,a_0)}$. Hence, the pair $(\rho^*,Q^*)$ solves the Q-AROE,
where
$\rho^* := LQ^*(x_0,a_0)$.
Moreover, the optimal value $\rho^*$ can be obtained as the limit of
iterate differences:
\[
\rho^* = \lim_{n\to\infty}
\big(L^{n+1}Q_0(x,a)-L^{n}Q_0(x,a)\big)
\]
for any initial $Q_0$ and any $(x,a)\in S\times A$.
\end{corollary}

The proof is given in Appendix~\ref{app:cor:Lhat_banach}.

\section{Entropy-Regularized MDPs and the Soft Q-AROE}
\label{sec:entropy}

In the Q-AROE~\eqref{eq:Q_ACOE} the agent commits at each state to a single best action through a hard maximum. This is enough for optimality, but the hard maximum is discontinuous in the problem data, it leaves no room for exploration in a learning setting, and it does not provide the Lipschitz dependence of the best response on the population distribution that the fixed-point analysis of mean-field games requires. Entropy regularization addresses these three issues at once \citep{GeScPi19,NeJoGo17}: the hard maximization is replaced by a maximization over randomized actions with a bonus of $\delta$ times the Shannon entropy of the action distribution. This is a strictly concave program on the simplex whose unique solution is the softmax (Boltzmann) distribution and whose optimal value is the log-sum-exp, a smooth surrogate of the maximum. The temperature $\delta>0$ controls the trade-off between reward and randomization, and the hard maximization is recovered as $\delta\downarrow0$. In this section we introduce the resulting soft Q-AROE, show that its operator is a span-seminorm contraction at the same rate and under the same condition as in Section~\ref{sec:classical}, and derive the Lipschitz properties of the softmax that drive the analysis of Section~\ref{sec:mfg}. The proofs for this section are collected in Appendix~\ref{app:entropy}.

\subsection{The entropy-regularized problem and the softmax}
\label{sec:variational}

For a finite set $E$ and $\pi\in\Delta(E)$, the \emph{Shannon entropy} is $$H(\pi):=-\sum_{a\in E}\pi(a)\log\pi(a),$$ with $0\log0:=0$; it satisfies $0\le H(\pi)\le\log|E|$, is strictly concave, and its gradient diverges at the boundary of the simplex \citep{CoTh06}. Let $\delta>0$ be the regularization parameter (temperature). The entropy-regularized one-stage reward of a policy $\pi\in\Pi$ at state $x$ is $$r^\pi(x)+\delta H(\pi(\cdot\mid x))=\sum_{a\in A(x)}\pi(a\mid x)\bigl(r(x,a)-\delta\log\pi(a\mid x)\bigr),$$ and the entropy-regularized average reward of $\pi$ is
\[
\rho_\delta^\pi(x)
:=
\liminf_{N\to\infty}
\frac{1}{N}
\E_x^\pi\!\left[
\sum_{t=0}^{N-1}
\left(
r(X_t,A_t)
-
\delta\log\pi(A_t\mid X_t)
\right)
\right].
\]
The classical problem is recovered as $\delta\downarrow 0$, while for $\delta>0$ the agent is rewarded for randomizing. The building block of the theory is the following variational identity, sometimes called the Gibbs variational principle; see, e.g., \citet[Proposition~1.4.2]{DuEl97} or \citet{GaPa18}.

\begin{proposition}[Softmax best response]
\label{prop:softmax}
Let $E$ be a finite set, $\delta>0$ and $q\in\mathbb{R}^E$. The problem
\begin{equation}
\label{eq:ent_var}
\max_{\pi\in\Delta(E)}
\Bigl\{
\sum_{a\in E}\pi(a)\,q(a)
+
\delta H(\pi)
\Bigr\}
\end{equation}
has the unique maximizer
\begin{equation}\label{eq:softmax}
\sigma_\delta(q)(a)
:=
\frac{\exp(q(a)/\delta)}
     {\sum_{b\in E}\exp(q(b)/\delta)},
\qquad a\in E,
\end{equation}
and its optimal value is the log-sum-exp
$ \displaystyle
\lse_\delta(q):=\delta\log\Bigl(
\sum_{a\in E}\exp(q(a)/\delta)
\Bigr).
$
\end{proposition}

\begin{remark}[Properties of the log-sum-exp]
\label{rem:lse}
The map $q\mapsto\lse_\delta(q)$ is sometimes itself called the ``softmax,'' though we reserve that name for the distribution~\eqref{eq:softmax}. It is a smooth convex function whose gradient is $\sigma_\delta$ (Lemma~\ref{lem:hessian_lse} below), it is monotone in $q$, it is shift-equivariant, $\lse_\delta(q+c\1)=\lse_\delta(q)+c$ for $c\in\R$, while $\sigma_\delta(q+c\1)=\sigma_\delta(q)$, and it satisfies the two-sided bound
\[
\max_{a}q(a)\le\lse_\delta(q)\le\max_{a}q(a)+\delta\log|E| .
\]
Hence $\lse_\delta$ converges to the hard maximum uniformly as $\delta\downarrow 0$.
\end{remark}

\subsection{The soft Q-AROE and its span-seminorm solution}
\label{sec:soft_bellman_eqn}

Replacing the hard maximum in the Q-AROE~\eqref{eq:Q_ACOE} by the entropy-regularized maximization~\eqref{eq:ent_var} over $\pi(\cdot\mid y)\in\Delta(A(y))$ gives, by Proposition~\ref{prop:softmax}, the \emph{soft Q-AROE}
\begin{equation}\label{eq:soft_Q_AROE}
\rho+Q(x,a)
=r(x,a)+\sum_{y\in S}P(y\mid x,a)\,\delta\log\!\sum_{b\in A}\exp\!\left(\frac{Q(y,b)}{\delta}\right),
\qquad (x,a)\in S\times A,
\end{equation}
with the associated stochastic policy
\begin{equation}\label{eq:soft_policy}
\pi^{*}_\delta(a\mid y)
=
\sigma_\delta\bigl(Q(y,\cdot)\bigr)(a)=
\frac{\exp(Q(y,a)/\delta)}
     {\sum_{b\in A}\exp(Q(y,b)/\delta)} .
\end{equation}
Repeating the classical verification argument with the one-stage reward $r(x,a)-\delta\log\pi(a\mid x)$ and using Proposition~\ref{prop:softmax} shows that if $(\rho_\delta^*,Q_\delta^*)$ solves~\eqref{eq:soft_Q_AROE}, then $\rho_\delta^*=\sup_{\pi\in\Pi}\rho^\pi_\delta(x)$ for every $x\in S$ and the softmax policy~\eqref{eq:soft_policy} built on $Q^*_\delta$ is optimal for the entropy-regularized problem; see also \citet{GeScPi19,NeJoGo17}. Compared with~\eqref{eq:Q_ACOE}, the only change is that the hard $\max_b$ has been replaced by the smooth $\delta\log\sum_b\exp(\cdot/\delta)$, so that the right-hand side is now an analytic function of $Q$ and the optimal policy is given in closed form. As in the classical case, $Q^*_\delta$ is determined only up to an additive constant, which does not affect the policy~\eqref{eq:soft_policy} since the softmax is shift-invariant. The contraction argument of Section~\ref{sec:classical} carries over verbatim, because the log-sum-exp inherits the key property of the hard maximum that drives it.

\begin{lemma}[Log-sum-exp is $1$-Lipschitz in sup-norm]
\label{lem:lse_lip}
Let $E$ be a finite set. For any $u,v\in\mathbb{R}^E$, $\delta>0$,
\[
\left|\delta\log\sum_{a\in E}\exp(u(a)/\delta)
-\delta\log\sum_{a\in E}\exp(v(a)/\delta)\right|
\le\max_{a\in E}|u(a)-v(a)|.
\]
\end{lemma}

The proof is given in Appendix~\ref{app:lem:lse_lip}.

\begin{definition}[Soft Q-operator]
\label{def:soft_Q_operators}
For $Q:S\times A\to\mathbb{R}$ and $(x,a)\in S\times A$,
\begin{equation*}
(L_\delta Q)(x,a)
:=r(x,a)+\sum_{y\in S}\delta\log\!\sum_{b\in A}\exp(Q(y,b)/\delta)\,P(y\mid x,a) .
\label{eq:L_delta}
\end{equation*}
In the notation of Proposition~\ref{prop:softmax}, $(L_\delta Q)(x,a)= r(x,a)+\sum_{y\in S} P(y\mid x,a)\,\lse_\delta\bigl(Q(y,\cdot)\bigr)$.
\end{definition}

Again $L_\delta(Q+c\1)=L_\delta Q+c\1$ for every $c\in\R$, and the soft Q-AROE reads $\rho\1+Q=L_\delta Q$.

\begin{theorem}[$L_\delta$ is a span-seminorm contraction]
\label{thm:Lhat_delta_span}
Under Assumption~\ref{def:minor}, %
\[
\spn\!\left(L_\delta Q-L_\delta Q'\right)
\le(1-\varepsilon)\,\spn(Q-Q'), \quad Q,Q':S\times A\to\mathbb{R}.
\]
\end{theorem}

The proof is given in Appendix~\ref{app:thm:Lhat_delta_span}.

Recall that $\mathbb{R}^{S\times A}_{(x_0,a_0)}$ denotes the set of functions
$Q:S\times A\to\mathbb{R}$ such that
$Q(x_0,a_0)=0$ for a fixed $(x_0,a_0)\in S\times A$, a complete metric space with respect to the span seminorm, and define the normalized operator
$\widetilde{L}_{\delta}Q(x,a)
:=
L_{\delta}Q(x,a)
-
L_{\delta}Q(x_0,a_0)$.

\begin{corollary}[Solution to the soft Q-AROE via $L_\delta$: soft relative value iteration]
\label{cor:Lhat_soft_banach}
There exists a unique fixed point $Q_\delta^*$ of $\widetilde{L}_\delta$
in $\mathbb{R}^{S\times A}_{(x_0,a_0)}$. Consequently, the pair
$(\rho_\delta^*,Q_\delta^*)$ solves the soft Q-AROE, where
$
\rho_\delta^*
:=
L_\delta Q_\delta^*(x_0,a_0).
$
Moreover, the optimal average reward can be obtained from the iterate differences:
\[
\rho_\delta^*
=
\lim_{n\to\infty}
\left(
L_\delta^{\,n+1}Q_0(x,a)
-
L_\delta^{\,n}Q_0(x,a)
\right)
\]
for every initial condition $Q_0$ and every $(x,a)\in S\times A$.
\end{corollary}

The proof is given in Appendix~\ref{app:cor:Lhat_soft_banach}.

\begin{remark}[Computational advantage]
The structure is identical to the unregularized case: same contraction factor $1-\varepsilon$, same full kernel $P$, and the same relative value iteration, with $\max_b$ replaced by $\delta\log\sum_b\exp(\cdot/\delta)$. Beyond the smoothness gained, the regularization has a computational benefit: each iteration of soft relative value iteration is a closed-form evaluation of a log-sum-exp, and the optimal policy is obtained in closed form from~\eqref{eq:soft_policy}, without any search over actions. In Section~\ref{sec:learning} the same decomposition $P=\widehat P+\nu$ that drives the proof of Theorem~\ref{thm:Lhat_delta_span} will be used, in the analysis only, to obtain finite-sample bounds for a sample-based version of $L_\delta$.
\end{remark}

\subsection{Lipschitz properties of the softmax}
\label{sec:softmax_lip}

The mean-field analysis of Section~\ref{sec:mfg} requires that the softmax policy depend Lipschitz continuously on its input vector, with an explicit constant. The map $m\mapsto\pi^*_m$ studied there is the composition of $m\mapsto Q^*_m$ (Lipschitz with a constant determined by the regularity of the model) and of the softmax (Lipschitz with constant $1/\delta$), so that the temperature controls the sensitivity of the best response and, for $\delta$ large enough, makes the mean-field operator a contraction.

We now state the result formally. Let
$\sigma_\delta:\mathbb{R}^{n}\to\Delta^{n-1}$ be the softmax map (where
$n:=|A|$ for short and $\Delta^{n-1}$ denotes the simplex in $\R^n$)
\[
\sigma_\delta(z)(a):=\frac{\exp(z(a)/\delta)}{\sum_{b=1}^n\exp(z(b)/\delta)}.
\]
By Proposition~\ref{prop:softmax}, $\sigma_\delta$ is the optimizer of
the entropy-regularized inner problem~\eqref{eq:ent_var}, and it is
also the gradient of the log-sum-exp function:
\[
\sigma_\delta(z)
=
\nabla_z\!\left[\delta\log\sum_{b=1}^n\exp(z(b)/\delta)\right].
\]
This identification---the softmax as the gradient of a convex
function---is the structural reason for its Lipschitz continuity: we compute the Hessian of the log-sum-exp, bound its operator norm, and deduce the Lipschitz bound for the gradient via the mean-value theorem. Similar results are well known \citep{GaPa18}; the short derivation is included because the $\ell_\infty\to\ell_1$ form of the bound is exactly what the mean-field analysis needs (not the $\ell_2\to\ell_2$ form as in \citep{GaPa18}).

\begin{lemma}[Hessian computation]
\label{lem:hessian_lse}
Recall $\lse_\delta(z):=\delta\log\sum_{b=1}^n\exp(z(b)/\delta)$ for
$z\in\mathbb{R}^n$ and $\delta>0$. Then $\lse_\delta$ is twice
continuously differentiable, with gradient
$\nabla\lse_\delta(z)=\sigma_\delta(z)$
and Hessian
\begin{equation}\label{eq:hessian_form}
\nabla^2\lse_\delta(z)
=
\frac{1}{\delta}\bigl(\diag(\sigma_\delta(z))
-\sigma_\delta(z)\sigma_\delta(z)^\top\bigr).
\end{equation}
\end{lemma}

The proof is given in Appendix~\ref{app:lem:hessian_lse}.

\begin{remark}[Probabilistic interpretation of the Hessian]
The matrix $M(p):=\diag(p)-pp^\top$ is the \emph{covariance
matrix} of a vector-valued random variable distributed according to
$p\in\Delta^{n-1}$. Indeed, if $X\in\{e_1,\ldots,e_n\}$ takes value
$e_a$ with probability $p(a)$ (where $e_a$ is the $a$-th standard
basis vector), then $\mathbb{E}[X]=p$ and
$\Cov(X)=\mathbb{E}[XX^\top]-pp^\top
=\diag(p)-pp^\top=M(p)$.
Covariance matrices are always symmetric positive semi-definite, which
already gives us $\nabla^2\lse_\delta(z)\succeq 0$ for free.
This makes $\lse_\delta$ a convex function.
\end{remark}

The next step is to bound the $\ell_\infty \to \ell_1$ operator norm of the Hessian.

\begin{lemma}[$\ell_\infty \to \ell_1$ operator norm of the covariance]
\label{lem:opnorm_M_linf_l1}
For every $p\in\Delta^{n-1}$, let
$
M(p):=\diag(p)-pp^\top.
$
Then
$\|M(p)\|_{\infty\to 1}\le 1$,
where
$\|M(p)\|_{\infty\to 1}
:=
\sup_{\|v\|_\infty\le 1}\|M(p)v\|_1$.
\end{lemma}

The proof is given in Appendix~\ref{app:lem:opnorm_M_linf_l1}.

\begin{corollary}[$\ell_\infty\to\ell_1$ operator norm of the Hessian]
\label{cor:hessian_opnorm}
For every $z\in\mathbb{R}^n$ and $\delta>0$,
$\bigl\|\nabla^2\lse_\delta(z)\bigr\|_{\infty\to 1}
\le
1/\delta$.
\end{corollary}

The proof is given in Appendix~\ref{app:cor:hessian_opnorm}.

We now combine the Hessian bound with the mean-value theorem to obtain
a Lipschitz estimate for $\sigma_\delta=\nabla\lse_\delta$ in
the $\ell_\infty\to\ell_1$ setting.

\begin{lemma}[Lipschitz bound in $\|\cdot\|_\infty$ and $\|\cdot\|_1$]
\label{lem:softmax_linf_l1}
For every $\delta>0$ and $z,z'\in\mathbb{R}^n$,
\[
\|\sigma_\delta(z)-\sigma_\delta(z')\|_1
\le
\frac{1}{\delta}\,\|z-z'\|_\infty .
\]
Moreover, since $\sigma_\delta(z+c\1)=\sigma_\delta(z)$ for every $c\in\R$,
\begin{equation}\label{eq:softmax_span_lip}
\|\sigma_\delta(z)-\sigma_\delta(z')\|_1
\le
\frac{1}{\delta}\,\inf_{c\in\mathbb{R}}\|z-z'-c\1\|_\infty
= \frac{\spn(z-z')}{2\delta}.
\end{equation}
\end{lemma}

The proof is given in Appendix~\ref{app:lem:softmax_linf_l1}.

In mean-field games, one typically requires Lipschitz continuity with respect to the $\ell_1$-norm on the space of probability measures and the $\ell_\infty$-norm on the space of Q-functions. This is precisely why we carried out the above analysis in the $\ell_\infty \to \ell_1$ framework.

\section{Entropy-Regularized Mean-Field Games}
\label{sec:mfg}

Mean-field games (MFGs) extend MDPs to a large population of symmetric agents whose environment is shaped by the aggregate distribution of their states: a representative agent optimizes against a population distribution $m$, and equilibrium requires that $m$ coincide with the stationary distribution generated when all agents use the resulting optimal policy. In this section we formulate the entropy-regularized mean-field game under the average-reward criterion and show that the equilibrium is the fixed point of a contraction on $\Delta(S)$ when the temperature is large enough. The contraction yields existence and uniqueness, but its main role is to provide a provably convergent algorithm, which is made model-free in Section~\ref{sec:learning}. The idea that entropy regularization restores the regularity that classical finite mean-field games lack has been used in the discounted and finite-horizon settings \citep{CuKo21,guan2022,AnKaSa23}; in particular, \citet{guan2022} construct a contractive entropy-regularized mean-field operator for finite-horizon games with decoupled dynamics and a Kullback--Leibler regularizer. Here the dynamics depend on the population, the horizon is infinite, and the contraction constant is governed by the minorization mass $\varepsilon$ rather than by the horizon length.

Following \citet{AnKaSa23}, we impose Lipschitz regularity on the reward and the kernel that is uniform in the population distribution and in the state and action arguments (measured in the discrete metric). Two consequences of this uniformity drive the analysis: the policy-averaged kernel is $K_1$-Lipschitz in the action distribution, and the span of the reward is bounded by $2L_1$, so that the whole analysis is expressed through the Lipschitz constants $K_1$ and $L_1$, the minorization mass $\varepsilon$ and the temperature $\delta$.

\subsection{Finite-state mean-field model}

Recall that $S$ is a finite state space, $A$ is a finite action space, and
$\Delta(S)$ is the probability simplex over $S$. A \emph{mean-field
game model} is specified by a controlled transition kernel
    $P(\cdot\mid x,a,m)\in\Delta(S)$ depending on the population
    distribution $m\in\Delta(S)$, and a one-stage reward $r(x,a,m)\in\mathbb{R}$ also depending on
    $m$.

Given a stationary Markov stochastic policy $\pi\in\Pi$ and a
population distribution $m\in\Delta(S)$, the policy-averaged kernel
and reward~\eqref{eq:induced} extend to the mean-field
model as
\begin{equation}
\label{eq:policy_averaged_mfg}
P_m^\pi(y\mid x)
:=
\sum_{a\in A}\pi(a\mid x)\,P(y\mid x,a,m),
\qquad
r_m^\pi(x)
:=
\sum_{a\in A}\pi(a\mid x)\,r(x,a,m).
\end{equation}
The entropy-regularized one-stage reward of
Section~\ref{sec:variational} is then
$r_m^\pi(x)+\delta H(\pi(\cdot\mid x))$; in the terminology of
\citet{AnKaSa23}, this corresponds to the regularizer
$\Omega(\pi(\cdot\mid x)) := -\delta H(\pi(\cdot\mid x))$, which is
$\delta$-strongly convex on $\Delta(A)$. We do not need the general
strongly convex machinery of \citet{AnKaSa23}, since the explicit
softmax/log-sum-exp analysis of Section~\ref{sec:entropy} gives the
sharper, dimension-free Lipschitz constant $1/\delta$
(Lemma~\ref{lem:softmax_linf_l1}) in place of the generic
$\sqrt{|A|}/\rho$ bound obtained there via Legendre--Fenchel duality.

\subsection{Mean-field equilibrium}

In the stationary (time-homogeneous) setting the population
distribution is taken to be constant in time, $m_t\equiv m$.

\begin{definition}[Mean-field equilibrium]
\label{def:mfe}
A pair $(\pi^{*},m^{*})\in\Pi\times\Delta(S)$ is a \emph{stationary
mean-field equilibrium} if:
\begin{enumerate}
  \item \textbf{Optimality.} Given $m^{*}$, the policy $\pi^{*}$ is
    optimal for the entropy-regularized average-reward MDP with
    transition kernel $P(\cdot\mid x,a,m^{*})$ and reward
    $r(x,a,m^{*})$.
  \item \textbf{Consistency.} $m^{*}$ is an invariant distribution of
    the induced Markov chain:
    \[
    m^{*}(y)
    =
    \sum_{x\in S}m^{*}(x)\,P_{m^{*}}^{\pi^{*}}(y\mid x),
    \qquad y\in S.
    \]
\end{enumerate}
\end{definition}

In stationary mean-field games, for each $m\in\Delta(S)$, the representative agent faces an MDP whose
entropy-regularized AROE admits, by Proposition~\ref{prop:softmax},
the closed-form softmax policy
\begin{equation}\label{eq:mfg_softmax}
\pi_m^{*}(a\mid x)
=
\frac{\exp(Q_m^{*}(x,a)/\delta)}
     {\sum_{b\in A}\exp(Q_m^{*}(x,b)/\delta)},
\end{equation}
where $Q_m^{*}$ is the (unique up to an additive constant) fixed point
of the soft Q-operator $L_{\delta,m}$ of
Definition~\ref{def:soft_Q_operators} with data
$(P(\cdot\mid\cdot,\cdot,m), r(\cdot,\cdot,m))$, solving the soft
Q-AROE~\eqref{eq:soft_Q_AROE} at population distribution $m$. Since the
softmax map is invariant under adding a constant to all entries of its
argument, the policy $\pi_m^{*}$ in~\eqref{eq:mfg_softmax} is
well-defined despite the additive-constant ambiguity in $Q_m^{*}$.

\begin{definition}[Mean-field operator]
\label{def:mfe_op}
Define the map $\mathcal{F}:\Delta(S)\to\Delta(S)$ by
\begin{equation}\label{eq:mfe_op}
\mathcal{F}(m)(y)
:=
\sum_{x\in S}P_m^{\pi_m^{*}}(y\mid x)\,m(x)
=
\sum_{x\in S}\sum_{a\in A}P(y\mid x,a,m)\,\pi_m^{*}(a\mid x)\,m(x),
\qquad y\in S.
\end{equation}
A pair $(\pi^{*},m^{*})$ with $\pi^{*}=\pi_{m^{*}}^{*}$ is a mean-field
equilibrium if and only if $m^{*}$ is a fixed point of
$\mathcal{F}$: the consistency condition of Definition~\ref{def:mfe} is
the above fixed-point identity, while the optimality condition holds
because $\pi_{m^{*}}^{*}$ is an optimal policy of the
entropy-regularized average-reward MDP at population $m^{*}$. Hence, in the remainder of this section, our goal is to compute the fixed point of $\mathcal{F}$.
\end{definition}

\subsection{Lipschitz assumptions}
\label{sec:mfg_lip}

We now state the standing regularity assumptions of this section,
in the form used by \citet{AnKaSa23}. Write
$\|\mu\|_1:=\sum_x|\mu(x)|$ for $\mu$ a signed measure on $S$, and let
$\1_{\{\cdot\}}$ denote the indicator.

\begin{assumption}[Lipschitz kernel and reward, uniform in $(x,a)$]
\label{ass:lip_primitives}
There exist constants $K_1,L_1\ge 0$ such that for every
$x,\hat{x}\in S$, $a,\hat{a}\in A$, and $m,\hat{m}\in\Delta(S)$,
\begin{align}
\bigl\|P(\cdot\mid x,a,m)-P(\cdot\mid \hat{x},\hat{a},\hat{m})\bigr\|_1
&\le
K_1\Bigl(\1_{\{x\neq\hat{x}\}} + 2\cdot\1_{\{a\neq\hat{a}\}}
+ \|m-\hat{m}\|_1\Bigr),
\label{eq:lip_kernel}\\[2pt]
\bigl|r(x,a,m)-r(\hat{x},\hat{a},\hat{m})\bigr|
&\le
L_1\Bigl(\1_{\{x\neq\hat{x}\}} + \1_{\{a\neq\hat{a}\}}
+ \|m-\hat{m}\|_1\Bigr).
\label{eq:lip_reward}
\end{align}
\end{assumption}

\begin{assumption}[Uniform minorization]
\label{ass:minor_mfg}
There exists a sub-probability measure $\nu$ on $S$ with total mass
$\varepsilon\in(0,1)$ such that
$P(y\mid x,a,m)\ge\nu(y)$ for every $x\in S$, $a\in A$, $y\in S$
and every $m\in\Delta(S)$.
\end{assumption}

\noindent Note that Assumption~\ref{ass:minor_mfg} allows us to apply the results established in Sections~\ref{sec:classical} and~\ref{sec:entropy} uniformly in $m$.

The following bound on the reward span, an immediate consequence of the
uniformity in Assumption~\ref{ass:lip_primitives}, will be used in
Lemma~\ref{lem:span-Qm}.

\begin{lemma}[Reward-span bound]
\label{lem:span_r}
Under Assumption~\ref{ass:lip_primitives}, for every $m\in\Delta(S)$,
\begin{equation}
\label{eq:span_r}
\spn\bigl(r(\cdot,\cdot,m)\bigr)
:=\max_{(x,a)\in S\times A}r(x,a,m)-\min_{(x,a)\in S\times A}r(x,a,m)
\le 2L_1 .
\end{equation}
\end{lemma}

\begin{proof}{Proof.}
Fix $m\in\Delta(S)$. Since $S\times A$ is finite, the maximum and minimum
of $r(\cdot,\cdot,m)$ are attained; let $(x_1,a_1)$ and $(x_2,a_2)$ be a
maximizer and a minimizer, respectively, so that
$\spn\bigl(r(\cdot,\cdot,m)\bigr) = r(x_1,a_1,m) - r(x_2,a_2,m)$.
Inserting the intermediate point $(x_2,a_1)$ and applying the triangle
inequality,
\[
r(x_1,a_1,m) - r(x_2,a_2,m)
= \bigl(r(x_1,a_1,m) - r(x_2,a_1,m)\bigr)
+ \bigl(r(x_2,a_1,m) - r(x_2,a_2,m)\bigr).
\]
We bound each difference using the reward part~\eqref{eq:lip_reward} of
Assumption~\ref{ass:lip_primitives} with $\hat{m}=m$, so that the population
term $\|m-\hat{m}\|_1$ vanishes. In the first difference only the state
changes ($a$ is fixed at $a_1$), so it is at most
$L_1\,\1_{\{x_1\neq x_2\}}\le L_1$; in the second only the action changes
($x$ is fixed at $x_2$), so it is at most $L_1\,\1_{\{a_1\neq a_2\}}\le L_1$.
Adding the two bounds gives
$\spn\bigl(r(\cdot,\cdot,m)\bigr)\le L_1 + L_1 = 2L_1$,
as claimed.\Halmos
\end{proof}

\begin{lemma}[Averaging against mean-zero weights]
\label{lem:span-tv}
Let $E$ be a finite set, $(V,\|\cdot\|)$ a normed space,
$g:E\to V$, and $\varphi:E\to\mathbb{R}$ with
$\sum_{e\in E}\varphi(e)=0$. Then
\[
\Bigl\|\sum_{e\in E} g(e)\,\varphi(e)\Bigr\|
\;\le\;
\frac{1}{2}\,
\max_{e,e'\in E}\bigl\|g(e)-g(e')\bigr\|
\;\|\varphi\|_1 .
\]
In particular, for $V=\mathbb{R}$ and $\mu,\nu\in\Delta(E)$,
\( \displaystyle
\Bigl|\sum_{e\in E} g(e)\bigl(\mu(e)-\nu(e)\bigr)\Bigr|
\le
\frac{\spn(g)}{2}\,\|\mu-\nu\|_1 .
\)
\end{lemma}

\begin{proof}{Proof.}
Decompose $\varphi=\varphi^{+}-\varphi^{-}$ into its positive and
negative parts. Since $\sum_e\varphi(e)=0$, the two parts have equal
mass $c:=\sum_e\varphi^{+}(e)=\sum_e\varphi^{-}(e)
=\tfrac{1}{2}\|\varphi\|_1$. If $c=0$ the claim is trivial; otherwise
$\alpha:=\varphi^{+}/c$ and $\beta:=\varphi^{-}/c$ are probability
vectors on $E$, and
\[
\sum_e g(e)\varphi(e)
= c\Bigl(\sum_e \alpha(e)\,g(e) - \sum_{e'} \beta(e')\,g(e')\Bigr)
= c\sum_{e,e'}\alpha(e)\beta(e')\bigl(g(e)-g(e')\bigr),
\]
using $\sum_{e'}\beta(e')=\sum_e\alpha(e)=1$. The triangle inequality
gives
\[
\Bigl\|\sum_e g(e)\varphi(e)\Bigr\|
\le c\,\max_{e,e'}\|g(e)-g(e')\|
= \frac{\|\varphi\|_1}{2}\,\max_{e,e'}\|g(e)-g(e')\| .
\]
The scalar case follows with $\varphi=\mu-\nu$ and
$\max_{e,e'}|g(e)-g(e')|=\spn(g)$.\Halmos
\end{proof}

A first consequence of the averaging lemma is the classical
contraction property of transition kernels acting on differences of
probability measures, which we will use for the measure-perturbation
term in Theorem~\ref{thm:F-contraction}.

\begin{corollary}[Dobrushin contraction; cf.\
{\citealp[Lemma~A.2]{KoRa08}}]
\label{cor:dobrushin}
Let $K(\cdot\mid x)$ be a transition kernel on the finite set $S$,
with \emph{Dobrushin coefficient}
\[
\theta_K
:=
\frac{1}{2}\max_{x,z\in S}
\bigl\|K(\cdot\mid x)-K(\cdot\mid z)\bigr\|_1 \;\in[0,1].
\]
Then, for every $\mu,\mu'\in\Delta(S)$,
\( \displaystyle
\|\mu K - \mu' K\|_1 \;\le\; \theta_K\,\|\mu-\mu'\|_1\),
where
\( \displaystyle (\mu K)(y):=\sum_{x\in S}K(y\mid x)\,\mu(x) \).
\
Moreover, if $K(\cdot\mid x)\ge\nu_0$ pointwise for every $x\in S$,
for some sub-probability measure $\nu_0$ on $S$ with total mass
$\varepsilon_0\in(0,1)$, then $\theta_K\le 1-\varepsilon_0$.
\end{corollary}

\begin{proof}{Proof.}
The function $\varphi:=\mu-\mu'$ satisfies $\sum_x\varphi(x)=0$, so
Lemma~\ref{lem:span-tv} with $E=S$ and
$g(x):=K(\cdot\mid x)\in(\mathbb{R}^S,\|\cdot\|_1)$ gives
\[
\|\mu K-\mu' K\|_1
=\Bigl\|\sum_{x\in S} g(x)\,\varphi(x)\Bigr\|_1
\le\frac{1}{2}\max_{x,z\in S}
\bigl\|K(\cdot\mid x)-K(\cdot\mid z)\bigr\|_1\,\|\varphi\|_1
=\theta_K\,\|\mu-\mu'\|_1 .
\]
For the second claim, write
$\widehat{K}(\cdot\mid x):=K(\cdot\mid x)-\nu_0\ge 0$, which has total
mass $1-\varepsilon_0$ for every $x$. Since the $\nu_0$-parts cancel,
$K(\cdot\mid x)-K(\cdot\mid z)
=\widehat{K}(\cdot\mid x)-\widehat{K}(\cdot\mid z)$, and the triangle
inequality gives
\[
\bigl\|K(\cdot\mid x)-K(\cdot\mid z)\bigr\|_1
\le\bigl\|\widehat{K}(\cdot\mid x)\bigr\|_1
+\bigl\|\widehat{K}(\cdot\mid z)\bigr\|_1
=2(1-\varepsilon_0),
\]
so $\theta_K\le 1-\varepsilon_0$.\Halmos
\end{proof}

The averaging lemma also converts the action-uniformity in
Assumption~\ref{ass:lip_primitives} into Lipschitz continuity of the
policy-averaged kernel~\eqref{eq:policy_averaged_mfg} with respect to
the action distribution prescribed by the policy. This is the
analogue of \citet[Proposition~1]{AnKaSa23}, with the optimal-coupling
argument there replaced by the elementary mean-zero decomposition
above. Only the kernel bound is needed in the sequel; the analogous
bound for the policy-averaged reward holds by the same argument but
plays no role in the contraction proof, so we omit it.

\begin{proposition}[Lipschitz continuity of the policy-averaged
kernel]
\label{prop:policy-avg-lip}
Under Assumption~\ref{ass:lip_primitives}, for every $x\in S$, every
pair of policies $\pi,\hat{\pi}\in\Pi$, and every
$m,\hat{m}\in\Delta(S)$,
\begin{equation}
\label{eq:policy_avg_kernel_lip}
\bigl\|P_m^{\pi}(\cdot\mid x)-P_{\hat{m}}^{\hat{\pi}}(\cdot\mid x)
\bigr\|_1
\le
K_1\Bigl(\bigl\|\pi(\cdot\mid x)-\hat{\pi}(\cdot\mid x)\bigr\|_1
+ \|m-\hat{m}\|_1\Bigr).
\end{equation}
\end{proposition}

\begin{proof}{Proof.}
Splitting
\begin{align*}
	P_m^{\pi}(\cdot\mid x)-P_{\hat{m}}^{\hat{\pi}}(\cdot\mid x)
	=
	\sum_{a\in A}
		\bigl(\pi(a\mid x)-\hat{\pi}(a\mid x)\bigr)\,
		P(\cdot\mid x,a,m)
	+
	\sum_{a\in A} \hat{\pi}(a\mid x)\,
		\bigl(P(\cdot\mid x,a,m)-P(\cdot\mid x,a,\hat{m})\bigr),
\end{align*}
for the second summand on the right-hand side, the triangle inequality and the $m$-part of
\eqref{eq:lip_kernel} give
\begin{align*}
\Bigl\|\sum_a \hat{\pi}(a\mid x)
\bigl(P(\cdot\mid x,a,m)-P(\cdot\mid x,a,\hat{m})\bigr)\Bigr\|_1
\le \sum_a \hat{\pi}(a\mid x)\,K_1\|m-\hat{m}\|_1
= K_1\|m-\hat{m}\|_1 ,
\end{align*}
and for the first summand, $\varphi(a):=\pi(a\mid x)-\hat{\pi}(a\mid x)$
satisfies $\sum_a\varphi(a)=0$, so Lemma~\ref{lem:span-tv} with
$g(a):=P(\cdot\mid x,a,m)\in(\mathbb{R}^S,\|\cdot\|_1)$ and the
action part of \eqref{eq:lip_kernel} (with $x=\hat{x}$, $m=\hat{m}$)
yield
\begin{align*}
\Bigl\|\sum_a \varphi(a)\,P(\cdot\mid x,a,m)\Bigr\|_1
&\le \frac{1}{2}\max_{a,\hat{a}}
\bigl\|P(\cdot\mid x,a,m)-P(\cdot\mid x,\hat{a},m)\bigr\|_1\,
\|\varphi\|_1\\
&\le \frac{1}{2}\cdot 2K_1\,\|\varphi\|_1
= K_1\,\bigl\|\pi(\cdot\mid x)-\hat{\pi}(\cdot\mid x)\bigr\|_1 .
\tag*{\Halmos}
\end{align*}
\end{proof}

\begin{remark}
The constant $K_1$ in front of the policy difference
in~\eqref{eq:policy_avg_kernel_lip} is where the action-uniformity of
Assumption~\ref{ass:lip_primitives} enters the analysis: without it,
the best available bound on the policy part is the trivial constant
$1$ (since $\sum_y P(y\mid x,a,m)=1$). It is this factor $K_1$ that
multiplies the policy-perturbation term in
Theorem~\ref{thm:F-contraction} below.
\end{remark}

\subsection{Lipschitz continuity of \texorpdfstring{$Q_m^*$ and $\pi_m^*$}{Q\_m* and pi\_m*}}
\label{sec:lipschitz}

The proof that $\mathcal{F}$ is a contraction proceeds through two
Lipschitz estimates: one for the soft Q-function $m \mapsto Q_m^*$ in
span seminorm, and one for the softmax policy $m \mapsto \pi_m^*$ in
$\ell^1$-norm. The strategy follows the perturbed fixed-point argument
of \citet{AnKaSa23}, adapted from the discounted-reward setting to the
entropy-regularized average-reward setting of this paper. The soft Q-operator
$L_{\delta,m}$ is a span-seminorm contraction uniformly in
$m$ by Theorem~\ref{thm:Lhat_delta_span} and
Assumption~\ref{ass:minor_mfg}:
\begin{equation}
\label{eq:Lhat-uniform-contraction}
\spn\bigl(L_{\delta,m}Q - L_{\delta,m}Q'\bigr)
\leq (1-\varepsilon)\,\spn(Q-Q'),
\qquad \forall m\in\Delta(S),\ \forall Q,Q' .
\end{equation}
Therefore, for each $m$, the operator $L_{\delta,m}$ has a unique (up to
additive constant) fixed point $Q_m^*$ in the span quotient, or equivalently the operator $\widetilde{L}_{\delta,m}Q(x,a) := L_{\delta,m}Q(x,a) - L_{\delta,m}Q(x_0,a_0)$ has a unique fixed point $Q_m^*$ in $\R^{S\times A}_{(x_0,a_0)}$, which
together with $\rho_m^* := L_{\delta,m} Q_m^*(x_0,a_0)$
solves the soft Q-AROE~\eqref{eq:soft_Q_AROE} at population
distribution $m$.

We first record a uniform span bound on $Q_m^*$, expressed through the
Lipschitz constant $L_1$ via~\eqref{eq:span_r}.

\begin{lemma}[Uniform span bound on $Q_m^*$]
\label{lem:span-Qm}
Under Assumptions~\ref{ass:lip_primitives} and~\ref{ass:minor_mfg}, for every
$m\in\Delta(S)$,
\[
\spn(Q_m^*)
\;\leq\;
\frac{\spn\bigl(r(\cdot,\cdot,m)\bigr)}{\varepsilon}
\;\leq\;
\frac{2L_1}{\varepsilon}.
\]
\end{lemma}

\begin{proof}{Proof.}
Write $V_m^*(y) := \lse_\delta\bigl(Q_m^*(y,\cdot)\bigr)$. The
soft Q-AROE~\eqref{eq:soft_Q_AROE} at population distribution $m$
reads
\[
\rho_m^* + Q_m^*(x,a)
= r(x,a,m) + \sum_{y\in S} P(y\mid x,a,m)\,V_m^*(y),
\qquad (x,a)\in S\times A.
\]
Fix $(x_1,a_1),(x_2,a_2)\in S\times A$ and subtract the corresponding
two equations; the constant $\rho_m^*$ cancels. Decomposing
$P(\cdot\mid x_i,a_i,m) = \widehat{P}(\cdot\mid x_i,a_i,m) + \nu$
(Assumption~\ref{ass:minor_mfg}), the $\nu$-parts of the two sums are
identical and cancel as well. Since each residual kernel is
non-negative with total mass $1-\varepsilon$,
\[
Q_m^*(x_1,a_1) - Q_m^*(x_2,a_2)
\leq \spn\bigl(r(\cdot,\cdot,m)\bigr)
+ (1-\varepsilon)\Bigl(\max_{y\in S} V_m^*(y)
- \min_{y\in S} V_m^*(y)\Bigr).
\]
Because the action set is the same at every state, Lemma~\ref{lem:lse_lip} applies
to the pair of vectors $Q_m^*(y_1,\cdot),\,Q_m^*(y_2,\cdot)\in
\mathbb{R}^{A}$ and yields, for any $y_1,y_2\in S$,
\[
\bigl|V_m^*(y_1) - V_m^*(y_2)\bigr|
\leq \max_{b\in A}\bigl|Q_m^*(y_1,b)-Q_m^*(y_2,b)\bigr|
\leq \spn(Q_m^*),
\]
i.e.\ $\spn(V_m^*) \leq \spn(Q_m^*)$. Taking the maximum over
$(x_1,a_1)$ and the minimum over $(x_2,a_2)$ in the previous display
and rearranging,
$\varepsilon\,\spn(Q_m^*) \leq \spn\bigl(r(\cdot,\cdot,m)\bigr)
\leq 2L_1$,
where the last inequality is~\eqref{eq:span_r}.\Halmos
\end{proof}

The next lemma is the heart of this section.

\begin{lemma}[Lipschitz soft Q-function]
\label{lem:lipschitz-Q}
Under Assumptions~\ref{ass:lip_primitives}
and~\ref{ass:minor_mfg}, for every $m,m'\in\Delta(S)$,
\[
\spn(Q_m^* - Q_{m'}^*) \leq C_Q\,\|m-m'\|_1,
\qquad
C_Q := \frac{2L_1}{\varepsilon}
\left(1 + \frac{K_1}{\varepsilon}\right).
\]
\end{lemma}

\begin{proof}{Proof.}
The proof has three steps: a one-step perturbation estimate, a
triangle inequality with the uniform
contraction~\eqref{eq:Lhat-uniform-contraction}, and an application of
Lemma~\ref{lem:span-Qm}.

\smallskip
\noindent\emph{Step 1 (one-step perturbation).}
We claim that for every $Q : S\times A\to\mathbb{R}$ and every
$m,m'\in\Delta(S)$,
\begin{equation}
\label{eq:one-step-perturbation}
\spn\bigl(L_{\delta,m} Q - L_{\delta,m'} Q\bigr)
\leq \bigl(2L_1 + K_1\,\spn(Q)\bigr)\,\|m-m'\|_1 .
\end{equation}
Fix $(x,a)\in S\times A$ and write
\begin{align*}
u(x,a) &:= r(x,a,m) + \sum_y P(y\mid x,a,m)\,
\lse_\delta(Q(y,\cdot)),\\
v(x,a) &:= r(x,a,m') + \sum_y P(y\mid x,a,m')\,
\lse_\delta(Q(y,\cdot)),
\end{align*}
so that $(L_{\delta,m} Q)(x,a) = u(x,a)$ and
$(L_{\delta,m'} Q)(x,a) = v(x,a)$. Then
\[
u(x,a) - v(x,a)
= \bigl(r(x,a,m) - r(x,a,m')\bigr)
+ \sum_y \bigl(P(y\mid x,a,m) - P(y\mid x,a,m')\bigr)\,
  \lse_\delta(Q(y,\cdot)).
\]
By the $m$-part of Assumption~\ref{ass:lip_primitives}, the reward
term is bounded in absolute value by $L_1\|m-m'\|_1$. For the kernel
term, $P(\cdot\mid x,a,m)$ and $P(\cdot\mid x,a,m')$ are
probability measures, so the scalar case of Lemma~\ref{lem:span-tv}
gives
\[
\left|\sum_y \bigl(P(y\mid x,a,m) - P(y\mid x,a,m')\bigr)\,
\lse_\delta(Q(y,\cdot))\right|
\leq \frac{\spn\bigl(\lse_\delta(Q(\cdot,\cdot))\bigr)}{2}\,
      K_1\|m-m'\|_1 .
\]
Since the action set is common to all states, Lemma~\ref{lem:lse_lip} applies
to the vectors $Q(y_1,\cdot),\,Q(y_2,\cdot)\in\mathbb{R}^A$ for any
$y_1,y_2\in S$ and yields
$\spn\bigl(\lse_\delta(Q(\cdot,\cdot))\bigr)\leq\spn(Q)$.
Combining the two bounds,
\[
|u(x,a) - v(x,a)|
\leq \left(L_1 + \frac{K_1\,\spn(Q)}{2}\right)\|m-m'\|_1
\qquad \text{uniformly in } (x,a).
\]
Since $L_{\delta,m} Q - L_{\delta,m'} Q = u - v$
and $\spn(u-v) \leq 2\max_{(x,a)}|(u-v)(x,a)|$, the
claim~\eqref{eq:one-step-perturbation} follows.

\smallskip
\noindent\emph{Step 2 (perturbed fixed-point inequality).}
By the soft Q-AROE, the fixed points satisfy
$L_{\delta,m} Q_m^* = Q_m^* + \rho_m^*\1$ and
$L_{\delta,m'} Q_{m'}^* = Q_{m'}^* + \rho_{m'}^*\1$. Since the span
seminorm is invariant under additive constants, the triangle
inequality in the span seminorm gives
\begin{align*}
	\spn(Q_m^* - Q_{m'}^*)
	= \spn\bigl(L_{\delta,m} Q_m^*
	- L_{\delta,m'} Q_{m'}^*\bigr)
	\leq \spn\bigl(L_{\delta,m} Q_m^*
	- L_{\delta,m'} Q_m^*\bigr)
	+ \spn\bigl(L_{\delta,m'} Q_m^*
	- L_{\delta,m'} Q_{m'}^*\bigr).
\end{align*}
By~\eqref{eq:one-step-perturbation} applied with $Q = Q_m^*$, the
first term is at most
$\bigl(2L_1 + K_1\,\spn(Q_m^*)\bigr)\|m-m'\|_1$. By the uniform
contraction~\eqref{eq:Lhat-uniform-contraction}, the second term is at
most $(1-\varepsilon)\,\spn(Q_m^* - Q_{m'}^*)$. Therefore
\[
\varepsilon\,\spn(Q_m^* - Q_{m'}^*)
\leq \bigl(2L_1 + K_1\,\spn(Q_m^*)\bigr)\|m-m'\|_1 .
\]

\smallskip
\noindent\emph{Step 3 (uniform bound and conclusion).}
By Lemma~\ref{lem:span-Qm}, $\spn(Q_m^*) \leq 2L_1/\varepsilon$.
Substituting,
\[
\spn(Q_m^* - Q_{m'}^*)
\leq \frac{1}{\varepsilon}
\left(2L_1 + \frac{2K_1 L_1}{\varepsilon}\right)\|m-m'\|_1
= \frac{2L_1}{\varepsilon}
\left(1 + \frac{K_1}{\varepsilon}\right)\|m-m'\|_1 . \Halmos
\]
\end{proof}

\begin{remark}
The constant $C_Q$ depends on $\varepsilon$, $K_1$, $L_1$, but
\emph{not} on the temperature $\delta$. This is because the
log-sum-exp operator is $1$-Lipschitz in its input regardless of
$\delta$ (Lemma~\ref{lem:lse_lip}). The temperature enters the
analysis only at the next step, when we pass from the Lipschitz
Q-function to the Lipschitz policy via the softmax.
\end{remark}

The Lipschitz continuity of the softmax policy now follows from
Lemma~\ref{lem:lipschitz-Q} and the softmax Lipschitz bound of
Section~\ref{sec:softmax_lip}. We exploit the shift-invariance of the
softmax to gain a factor $\tfrac{1}{2}$, which will propagate into the
final temperature threshold.

\begin{lemma}[Lipschitz softmax policy]
\label{lem:lipschitz-policy}
Under Assumptions~\ref{ass:lip_primitives}
and~\ref{ass:minor_mfg}, for every $x\in S$ and every
$m,m'\in\Delta(S)$,
\[
\|\pi_m^*(\cdot\mid x) - \pi_{m'}^*(\cdot\mid x)\|_1
\leq \frac{C_Q}{2\delta}\,\|m-m'\|_1 .
\]
\end{lemma}

\begin{proof}{Proof.}
By~\eqref{eq:mfg_softmax},
$\pi_m^*(\cdot\mid x) = \sigma_\delta(Q_m^*(x,\cdot))$. By the shift-invariant form~\eqref{eq:softmax_span_lip} of the softmax Lipschitz bound in
Lemma~\ref{lem:softmax_linf_l1},
\[
\|\sigma_\delta(z)-\sigma_\delta(z')\|_1
\leq \frac{\spn(z-z')}{2\delta}.
\]
Applying this with $z=Q_m^*(x,\cdot)$, $z'=Q_{m'}^*(x,\cdot)$ and
using $\spn\bigl(Q_m^*(x,\cdot)-Q_{m'}^*(x,\cdot)\bigr)
\leq\spn(Q_m^*-Q_{m'}^*)$ together with Lemma~\ref{lem:lipschitz-Q}
gives the claim. Note that the bound is well-posed despite the
additive-constant ambiguity of $Q_m^*$ and $Q_{m'}^*$, since both the
softmax and the span seminorm are shift-invariant.\Halmos
\end{proof}

\subsection{Banach contraction of the mean-field operator}
\label{sec:banach-contraction}

We are now in a position to prove that the mean-field operator
$\mathcal{F}:\Delta(S)\to\Delta(S)$ defined in~\eqref{eq:mfe_op} is a
strict contraction when the temperature $\delta$ is sufficiently
large. The argument splits the difference
$\mathcal{F}(m)(y) - \mathcal{F}(m')(y)$ into three pieces
corresponding to perturbations of the kernel, the policy, and the
population distribution, and bounds each piece in $\ell^1$-norm using
the estimates established so far.

\begin{theorem}[Contraction of $\mathcal{F}$]
\label{thm:F-contraction}
Under Assumptions~\ref{ass:lip_primitives}
and~\ref{ass:minor_mfg}, the mean-field operator
$\mathcal{F}:\Delta(S)\to\Delta(S)$ satisfies
\[
\|\mathcal{F}(m) - \mathcal{F}(m')\|_1
\leq C_{\mathcal{F}}\,\|m - m'\|_1,
\qquad \forall m,m'\in\Delta(S),
\]
where
\[
C_{\mathcal{F}} := K_1 + \frac{K_1\,C_Q}{2\delta}
+ 1 - \varepsilon,
\qquad
C_Q = \frac{2L_1}{\varepsilon}
\Bigl(1 + \frac{K_1}{\varepsilon}\Bigr).
\]
In particular, if $K_1 < \varepsilon$ and
\begin{equation}
\label{eq:delta-large}
\delta \;>\; \frac{K_1\,C_Q}{2(\varepsilon - K_1)}
\;=\;
\frac{K_1 L_1\,(\varepsilon + K_1)}
{\varepsilon^2(\varepsilon - K_1)},
\end{equation}
then $C_{\mathcal{F}} < 1$ and $\mathcal{F}$ is a strict contraction
on the complete metric space $(\Delta(S),\|\cdot\|_1)$.
\end{theorem}

\begin{proof}{Proof.}
Fix $m,m'\in\Delta(S)$. By the definition~\eqref{eq:mfe_op} of
$\mathcal{F}$ and the policy-averaged
notation~\eqref{eq:policy_averaged_mfg},
\[
\mathcal{F}(m)(y) - \mathcal{F}(m')(y)
=
\sum_{x\in S}\Bigl[
P_m^{\pi_m^{*}}(y\mid x)\,m(x)
- P_{m'}^{\pi_{m'}^{*}}(y\mid x)\,m'(x)
\Bigr].
\]
Adding and subtracting two intermediate terms decomposes the
summand:
\begin{multline*}
P_m^{\pi_m^{*}}(y\mid x)\,m(x)
- P_{m'}^{\pi_{m'}^{*}}(y\mid x)\,m'(x)
=
 \underbrace{\Bigl[P_m^{\pi_m^{*}}(y\mid x)
- P_{m'}^{\pi_m^{*}}(y\mid x)\Bigr]m(x)}_{\text{(I) kernel term}} \\
+
\underbrace{\Bigl[P_{m'}^{\pi_m^{*}}(y\mid x)
- P_{m'}^{\pi_{m'}^{*}}(y\mid x)\Bigr]m(x)}_{\text{(II) policy term}}
\;+\;
\underbrace{P_{m'}^{\pi_{m'}^{*}}(y\mid x)
\bigl[m(x) - m'(x)\bigr]}_{\text{(III) measure term}} .
\end{multline*}
Summing over $y\in S$ and applying the triangle inequality,
$\|\mathcal{F}(m) - \mathcal{F}(m')\|_1
\leq \mathrm{(I)} + \mathrm{(II)} + \mathrm{(III)}$,
where $\mathrm{(I)}$, $\mathrm{(II)}$, $\mathrm{(III)}$ denote the
$\ell^1$-norms of the corresponding sums over $x$. We bound each term
in turn.

\medskip
\noindent\emph{Term (I) --- kernel perturbation.}
By Proposition~\ref{prop:policy-avg-lip} with the same policy
$\pi=\hat{\pi}=\pi_m^{*}$ and populations $m,m'$,
\[
\mathrm{(I)}
\leq \sum_{x} m(x)\,
\bigl\|P_m^{\pi_m^{*}}(\cdot\mid x)
- P_{m'}^{\pi_m^{*}}(\cdot\mid x)\bigr\|_1
\leq K_1\|m-m'\|_1 \sum_{x} m(x)
= K_1\|m-m'\|_1 .
\]

\medskip
\noindent\emph{Term (II) --- policy perturbation.}
By Proposition~\ref{prop:policy-avg-lip} with the same population
$m=\hat{m}=m'$ and policies $\pi_m^{*},\pi_{m'}^{*}$, followed by
Lemma~\ref{lem:lipschitz-policy} (uniformly in $x$),
\begin{align*}
	\mathrm{(II)}
	\leq \sum_{x} m(x)\,
	\bigl\|P_{m'}^{\pi_m^{*}}(\cdot\mid x)
	- P_{m'}^{\pi_{m'}^{*}}(\cdot\mid x)\bigr\|_1
	\leq K_1 \sum_{x} m(x)\,
	\bigl\|\pi_m^{*}(\cdot\mid x)-\pi_{m'}^{*}(\cdot\mid x)\bigr\|_1
	\leq \frac{K_1 C_Q}{2\delta}\,\|m-m'\|_1 .
\end{align*}

\medskip
\noindent\emph{Term (III) --- measure perturbation.}
The policy-averaged kernel $P_{m'}^{\pi_{m'}^{*}}$ is exactly the
transition kernel on $S$ that a single agent experiences when the
population distribution is fixed at $m'$ and the agent uses the
equilibrium policy $\pi_{m'}^*$; for brevity write
$K_{m'}(y\mid x) := P_{m'}^{\pi_{m'}^{*}}(y\mid x)$. By
Assumption~\ref{ass:minor_mfg}, $P(\cdot\mid x,a,m') \geq \nu$
pointwise for every $a$; averaging over $\pi_{m'}^*(\cdot\mid x)$, the
same minorization is inherited by $K_{m'}$:
\[
K_{m'}(y\mid x)
= \sum_{a\in A}\pi_{m'}^{*}(a\mid x)\,P(y\mid x,a,m')
\geq \nu(y)\sum_{a\in A}\pi_{m'}^*(a\mid x) = \nu(y),
\quad \forall x,y\in S.
\]
Term (III) is precisely $\|m K_{m'} - m' K_{m'}\|_1$, so
Corollary~\ref{cor:dobrushin}, applied with the minorizing measure
$\nu$ of total mass $\varepsilon$, yields
\[
\mathrm{(III)}
= \|m K_{m'} - m' K_{m'}\|_1
\leq \theta_{K_{m'}}\,\|m-m'\|_1
\leq (1-\varepsilon)\,\|m-m'\|_1 .
\]

\medskip
\noindent\emph{Combining.} Adding the three bounds,
\[
\|\mathcal{F}(m) - \mathcal{F}(m')\|_1
\leq \Bigl(K_1 + \frac{K_1 C_Q}{2\delta}
+ 1-\varepsilon\Bigr)\|m-m'\|_1
= C_{\mathcal{F}}\,\|m-m'\|_1 .
\]
For $\mathcal{F}$ to be a strict contraction we need
$C_{\mathcal{F}} < 1$, i.e.\
$K_1 + K_1 C_Q/(2\delta) < \varepsilon$. If $K_1 < \varepsilon$,
this is equivalent to~\eqref{eq:delta-large}.\Halmos
\end{proof}

\begin{remark}[On the smallness condition]
\label{rem:smallness}
The contraction condition~\eqref{eq:delta-large} requires two things:
the kernel Lipschitz constant $K_1$ must be small relative to the
minorization mass ($K_1 < \varepsilon$), and the temperature $\delta$
must exceed $K_1 C_Q/(2(\varepsilon - K_1))$. The first condition is
structural: when the kernel depends too strongly on the population
distribution, the mean-field operator is genuinely non-contractive
and the equilibrium need not be unique. The second condition is the
trade-off purchased by entropy regularization: by smoothing the
policy at temperature $\delta$, its sensitivity to $m$ decays at rate
$1/\delta$, and choosing $\delta$ large compensates for the remaining
policy-sensitivity term in $C_{\mathcal{F}}$. Note that the
threshold~\eqref{eq:delta-large} scales with $K_1$: when the dynamics
are weakly coupled to the action and the population ($K_1$ small),
even small temperatures suffice, and the regularized equilibrium is
close to a hard-max equilibrium. As $\delta\downarrow 0$,
$C_{\mathcal{F}}$ blows up (unless $K_1 = 0$, in which case the
dynamics do not react to the policy perturbation at all) and the
Banach argument fails --- consistent with the well-known fact that
classical mean-field equilibria need not be unique.
\end{remark}

The contraction of $\mathcal{F}$ now yields existence, uniqueness, and
a constructive computational scheme via Banach iteration.

\begin{corollary}[Existence, uniqueness, and computability of the mean-field equilibrium]
\label{cor:MFNE}
Under Assumptions~\ref{ass:lip_primitives}
and~\ref{ass:minor_mfg}, suppose $K_1<\varepsilon$ and $\delta$
satisfies~\eqref{eq:delta-large}, so that $C_{\mathcal{F}} < 1$ in
Theorem~\ref{thm:F-contraction}. Then:
\begin{enumerate}
\item The map $\mathcal{F}$ has a unique fixed point $m^*\in\Delta(S)$. The canonical equilibrium policy $\pi^*:=\pi^*_{m^*}$ is given by the softmax~\eqref{eq:mfg_softmax}, and $(\pi^*,m^*)$ is a mean-field equilibrium. Every mean-field equilibrium $(\pi,m)$ satisfies $m=m^*$, and
\(
\pi(\cdot\mid x)=\pi^*(\cdot\mid x)\) whenever $m^*(x)>0$.
\item For any initial $m^{(0)} \in \Delta(S)$, the iterates
$m^{(k+1)} := \mathcal{F}(m^{(k)})$ converge to $m^*$ geometrically:
\[
\|m^{(k)} - m^*\|_1 \leq C_{\mathcal{F}}^{\,k}\,\|m^{(0)} - m^*\|_1 .
\]
\end{enumerate}
\end{corollary}

\begin{proof}{Proof.}
$(\Delta(S),\|\cdot\|_1)$
is a complete metric space (Section~\ref{sec:policies}). By Theorem~\ref{thm:F-contraction},
$\mathcal{F}$ is a strict contraction on this space, so Banach's
fixed point theorem yields a unique fixed point $m^* \in \Delta(S)$
with $\mathcal{F}(m^*) = m^*$ and geometric convergence of iterates at
rate $C_{\mathcal{F}}$. Setting $\pi^* := \pi_{m^*}^*$, the pair
$(\pi^*, m^*)$ satisfies both conditions of Definition~\ref{def:mfe} by
construction of $\mathcal{F}$: optimality holds because
$\pi^*_{m^*}$ is the soft optimal policy for the MDP at population
$m^*$, and consistency holds because $m^* = \mathcal{F}(m^*)$.

Let $(\pi,m)$ be any mean-field equilibrium, and let $(\rho,Q)$ solve the soft Q-AROE~\eqref{eq:soft_Q_AROE} at population $m$, so that $\rho$ is the optimal entropy-regularized average reward (Section~\ref{sec:soft_bellman_eqn}). Define
\[
d(x):=\lse_\delta(Q(x,\cdot))-\sum_{a\in A}\pi(a\mid x)\,Q(x,a)-\delta H(\pi(\cdot\mid x)),\qquad x\in S.
\]
By Proposition~\ref{prop:softmax}, $d(x)\ge0$, with equality precisely when $\pi(\cdot\mid x)=\pi^*_m(\cdot\mid x)$. By consistency, $m$ is invariant for $P^\pi_m$; by uniform minorization and Corollary~\ref{cor:dobrushin}, it is the unique invariant distribution, so the average reward of $\pi$ from every initial state equals its stationary reward under $m$. Averaging~\eqref{eq:soft_Q_AROE} against $m(x)\pi(a\mid x)$ and using $mP^\pi_m=m$ gives
\[
\sum_{x\in S}m(x)\bigl[r^\pi_m(x)+\delta H(\pi(\cdot\mid x))\bigr]=\rho-\sum_{x\in S}m(x)\,d(x).
\]
By optimality, the left-hand side equals $\rho$. Thus $d=0$ on the support of $m$, so $\pi=\pi^*_m$ there and
\[
\mathcal{F}(m)=mP^{\pi^*_m}_m=mP^\pi_m=m.
\]
Uniqueness of the fixed point yields $m=m^*$, and hence the asserted policy agreement.\Halmos
\end{proof}
\subsection{Algorithmic interpretation}
\label{sec:algorithm}

Corollary~\ref{cor:MFNE} yields a constructive algorithm to
compute the mean-field equilibrium:
\begin{enumerate}
  \item Initialize $m^{(0)}\in\Delta(S)$ arbitrarily.
  \item At iteration $k$, given $m^{(k)}$:
    \begin{enumerate}
      \item Solve the soft Q-AROE~\eqref{eq:soft_Q_AROE}
        with kernel $P(\cdot\mid\cdot,\cdot,m^{(k)})$ and reward
        $r(\cdot,\cdot,m^{(k)})$ to obtain
        $Q_{m^{(k)}}^{*}$ (up to an additive constant).
      \item Form the softmax policy $\pi_{m^{(k)}}^{*}$ in closed form
        via~\eqref{eq:mfg_softmax}.
      \item Compute $m^{(k+1)}=\mathcal{F}(m^{(k)})$ by pushing
        $m^{(k)}$ forward through the policy-averaged kernel
        $P_{m^{(k)}}^{\pi_{m^{(k)}}^{*}}$.
    \end{enumerate}
  \item Return $m^{*}=\lim_k m^{(k)}$ and
    $\pi^{*}=\pi_{m^{*}}^{*}$.
\end{enumerate}
Each step is computable in polynomial time in $|S|$ and $|A|$ (step 2a
is a Banach iteration in its own right, by
Corollary~\ref{cor:Lhat_soft_banach}). The overall algorithm converges
geometrically at rate $C_{\mathcal{F}}$. Steps 2a and 2c are the only ones that use the model $(P,r)$; in Section~\ref{sec:learning} they are replaced by soft Q-learning and by an empirical estimate of the kernel, respectively.

\section{Model-Free Learning of the Mean-Field Equilibrium}
\label{sec:learning}

The algorithm of Section~\ref{sec:algorithm} uses the model $(P,r)$ in exactly two places: to solve the soft Q-AROE at the current population (step~2a) and to push the population forward through the policy-averaged kernel (step~2c). In this section we replace both steps by sample-based counterparts. Throughout, the learner does not know $P$ or $r$ but has access to a \emph{simulator} of the mean-field environment, as in \citet{GuHuXuZh19} and \citet{AnKaSa23}: for any query $(x,a,m)\in S\times A\times\Delta(S)$ the simulator returns the reward $r(x,a,m)$ and an independent sample $Y\sim P(\cdot\mid x,a,m)$. The soft Q-function at a frozen population is then learned by a synchronous soft Q-learning scheme (Section~\ref{sec:softQL}), and the population update is performed with an empirical estimate of the transition kernel (Section~\ref{sec:mfg-learning}). The main results are a finite-sample bound for soft Q-learning in the span seminorm (Theorem~\ref{thm:softQL-rate}) and a convergence rate for the resulting model-free mean-field iteration (Theorem~\ref{thm:mfg-learning}).

Two features of the preceding theory make this possible with little additional effort. First, the softmax policy depends on $Q$ only through its equivalence class modulo constants, so it suffices to learn $Q^*_\delta$ in the span seminorm; this is precisely what relative value iteration, and its stochastic version below, delivers. Second, the minorization decomposition $P=\widehat P+\nu$ that drives all the contraction proofs of Sections~\ref{sec:classical}--\ref{sec:entropy} can be used once more, in the analysis only, to convert the span-seminorm contraction of $L_\delta$ into a supremum-norm contraction of an auxiliary operator, so that the finite-sample theory of \citet{ChMaShSh20} for norm-contractive stochastic approximation applies verbatim. The measure $\nu$ never enters the algorithm.

\subsection{Synchronous soft Q-learning at a fixed population}
\label{sec:softQL}

In this subsection the population is frozen and we consider a single entropy-regularized average-reward MDP $(S,A,P,r)$ with temperature $\delta>0$. We assume the minorization condition (Assumption~\ref{def:minor}) with minorizing measure $\nu$ of mass $\varepsilon\in(0,1)$, and we write $d:=|S||A|$; to avoid trivialities we assume $d\ge 2$. Let $(\rho^*_\delta,Q^*_\delta)$ be a solution of the soft Q-AROE~\eqref{eq:soft_Q_AROE}; $Q^*_\delta$ is unique up to an additive constant, and every statement below is invariant under $Q^*_\delta\mapsto Q^*_\delta+c\1$. For $Q\in\R^{S\times A}$ we use the shorthand
\[
V_Q(y):=\lse_\delta\bigl(Q(y,\cdot)\bigr)=\delta\log\sum_{b\in A}\exp\bigl(Q(y,b)/\delta\bigr),\qquad y\in S,
\]
so that $(L_\delta Q)(x,a)=r(x,a)+\sum_y P(y\mid x,a)V_Q(y)$, and we recall from Remark~\ref{rem:lse} and Lemma~\ref{lem:lse_lip} that $V_{Q+s\1}=V_Q+s$ for $s\in\R$ and that $|V_Q(y)-V_{Q'}(y)|\le\max_b|Q(y,b)-Q'(y,b)|$.

The synchronous sampling protocol is used: at every iteration the learner queries the simulator once at every state-action pair. The resulting algorithm is the entropy-regularized analogue of RVI Q-learning \citep{AbBeBo01}.

\begin{algorithm}[t]
\caption{Synchronous soft RVI Q-learning}
\label{alg:softQL}
\begin{algorithmic}[1]
\Require temperature $\delta>0$; initial $Q_0\in\R^{S\times A}$; step sizes $(\alpha_n)_{n\ge0}\subset(0,1]$; offset function $f:\R^{S\times A}\to\R$ (e.g.\ $f(Q)=Q(x_0,a_0)$, $f(Q)=\frac1d\sum_{x,a}Q(x,a)$, or $f\equiv 0$)
\For{$n=0,1,2,\dots$}
  \For{each $(x,a)\in S\times A$}
    \State query the simulator at $(x,a)$: observe $r(x,a)$ and $Y_n^{x,a}\sim P(\cdot\mid x,a)$
    \State form the target $\tau_n(x,a):=r(x,a)+\delta\log\sum_{b\in A}\exp\bigl(Q_n(Y_n^{x,a},b)/\delta\bigr)-f(Q_n)$
    \State $Q_{n+1}(x,a)\gets(1-\alpha_n)\,Q_n(x,a)+\alpha_n\,\tau_n(x,a)$
  \EndFor
\EndFor
\Ensure policy $\pi_n:=\sigma_\delta(Q_n)$, i.e.\ $\pi_n(a\mid x)=\exp(Q_n(x,a)/\delta)\big/\sum_{b}\exp(Q_n(x,b)/\delta)$
\end{algorithmic}
\end{algorithm}

We assume that the samples $\{Y_n^{x,a}:n\ge0,\ (x,a)\in S\times A\}$ are independent, with $Y_n^{x,a}\sim P(\cdot\mid x,a)$, and independent of $Q_0$. Let $\calF_n:=\sigma\bigl(Q_0,\,Y_j^{x,a}:\ j<n,\ (x,a)\in S\times A\bigr)$, so that $Q_n$ is $\calF_n$-measurable. In operator form, Algorithm~\ref{alg:softQL} reads
\begin{equation}\label{eq:softQL-operator-form}
Q_{n+1}=(1-\alpha_n)Q_n+\alpha_n\bigl[L_\delta Q_n+M_{n+1}-f(Q_n)\1\bigr],
\end{equation}
where
\[
M_{n+1}(x,a):=V_{Q_n}(Y_n^{x,a})-\sum_{y\in S}P(y\mid x,a)V_{Q_n}(y),\qquad (x,a)\in S\times A,
\]
is the sampling noise. The offset $f(Q_n)$ plays the role of the normalization $L_\delta Q(x_0,a_0)$ in soft relative value iteration (Corollary~\ref{cor:Lhat_soft_banach}): without it the iterates drift by roughly $\rho^*_\delta\sum_{j<n}\alpha_j$. For the quantities we care about, the span of the error and the softmax policy, the offset is irrelevant.

\begin{lemma}[The offset does not affect the span or the policy]
\label{lem:offset-invariance}
Let $(Q_n)$ and $(Q_n')$ be generated by Algorithm~\ref{alg:softQL} from the same $Q_0$, the same samples and the same step sizes, but with offset functions $f$ and $f'$, respectively. Then $Q_n-Q_n'\in\R\1$ for every $n\ge 0$. Consequently, $\spn(Q_n-Q)=\spn(Q_n'-Q)$ for every $Q\in\R^{S\times A}$ and $\sigma_\delta(Q_n)=\sigma_\delta(Q_n')$.
\end{lemma}

\begin{proof}{Proof.}
We argue by induction; the case $n=0$ is trivial. If $Q_n=Q_n'+s_n\1$ for some $s_n\in\R$, then $V_{Q_n}=V_{Q_n'}+s_n$, so that
\[
Q_{n+1}-Q_{n+1}'
=(1-\alpha_n)s_n\1+\alpha_n\bigl[s_n-f(Q_n)+f'(Q_n')\bigr]\1\in\R\1 .
\]
The remaining claims follow from the invariance of $\spn(\cdot)$ and of $\sigma_\delta$ under additive constants (Lemma~\ref{lem:softmax_linf_l1}).\Halmos
\end{proof}

By Lemma~\ref{lem:offset-invariance} we may analyze the un-normalized recursion ($f\equiv0$). The key observation is that, after subtracting a suitable scalar sequence, the error $Q_n-Q^*_\delta$ follows a stochastic approximation whose drift is a supremum-norm contraction. To see this, decompose $L_\delta$ according to $P=\widehat P+\nu$: for $Q\in\R^{S\times A}$ let
\begin{equation}\label{eq:hatL-def}
(\widehat L_\delta Q)(x,a):=r(x,a)+\sum_{y\in S}\widehat P(y\mid x,a)\,V_Q(y),
\qquad
c(Q):=\sum_{y\in S}\nu(y)\,V_Q(y),
\end{equation}
so that $L_\delta Q=\widehat L_\delta Q+c(Q)\1$. Note that $\widehat L_\delta$ is a $(1-\varepsilon)$-contraction in $\|\cdot\|_\infty$: for each $(x,a)$, $|(\widehat L_\delta Q-\widehat L_\delta Q')(x,a)|\le\sum_y\widehat P(y\mid x,a)|V_Q(y)-V_{Q'}(y)|\le(1-\varepsilon)\|Q-Q'\|_\infty$ by Lemma~\ref{lem:lse_lip}. This is the operator of the supremum-norm approach mentioned in Section~\ref{sec:span}; it cannot be used by the learner, since it involves $\nu$, but it can be used by the analyst.

Before stating the reduction we explain the idea behind it. Two obstacles prevent a
direct application of the stochastic approximation theory of \citet{ChMaShSh20} to the
error $Q_n-Q^*_\delta$: that theory requires the drift to be a contraction in a norm,
whereas $L_\delta$ contracts only in the span seminorm and does not contract constant
functions at all ($L_\delta(Q+s\1)=L_\delta Q+s\1$); and the error does not even
converge, since $L_\delta Q^*_\delta=Q^*_\delta+\rho^*_\delta\1$ implies that each step
pushes $Q_n$ by roughly $\alpha_n\rho^*_\delta\1$, so that $Q_n-Q^*_\delta$ grows along
$\1$ like $\rho^*_\delta\sum_{j<n}\alpha_j$. Both obstacles concern the one-dimensional
direction spanned by $\1$, which is invisible to the two quantities of interest, the span
of the error and the softmax policy. This suggests splitting the error as
\begin{equation}\label{eq:error-split}
Q_n-Q^*_\delta=\kappa_n\1+\widetilde E_n ,
\end{equation}
where the scalar $\kappa_n$ collects the component along $\1$, and choosing $\kappa_n$ so
that the remainder $\widetilde E_n$ follows a stochastic approximation whose drift is a
supremum-norm contraction; nothing is lost in doing so, because
$\spn(Q_n-Q^*_\delta)=\spn(\widetilde E_n)\le2\|\widetilde E_n\|_\infty$ and
$\sigma_\delta(Q_n)=\sigma_\delta(Q^*_\delta+\widetilde E_n)$ whatever the choice of
$\kappa_n$. The decomposition $L_\delta=\widehat L_\delta+c(\cdot)\1$ singles out the
right choice. Subtracting $Q^*_\delta$ from~\eqref{eq:softQL-operator-form} and using the
soft Q-AROE, the shift-equivariance of $L_\delta$ and this decomposition, one finds that
$Q_{n+1}-Q^*_\delta$ is the sum of
$(1-\alpha_n)\widetilde E_n+\alpha_n\bigl[\widehat L_\delta(Q^*_\delta+\widetilde E_n)-\widehat L_\delta Q^*_\delta\bigr]+\alpha_nM_{n+1}$
and of a multiple of $\1$ with coefficient
$\kappa_n+\alpha_n\bigl(\rho^*_\delta+c(Q^*_\delta+\widetilde E_n)-c(Q^*_\delta)\bigr)$.
Defining $\kappa_{n+1}$ to be this coefficient---that is, letting $\kappa_n$ absorb every
multiple of $\1$ that the recursion generates---leaves for $\widetilde E_{n+1}$ exactly
the first sum, which is a stochastic approximation driven by the martingale-difference
noise $M_{n+1}$ and with drift
$G(E):=\widehat L_\delta(Q^*_\delta+E)-\widehat L_\delta Q^*_\delta$, a
$(1-\varepsilon)$-contraction in $\|\cdot\|_\infty$ with fixed point $E=0$. The increment
of $\kappa_n$ has a transparent meaning: $\rho^*_\delta$ is the systematic drift that is
present even at the exact solution (the drift removed by the offset in relative value
iteration), and $c(Q^*_\delta+\widetilde E_n)-c(Q^*_\delta)$ is the $\nu$-average of the
effect of the current error on the continuation values, i.e.\ the part of the response of
$L_\delta$ to the error that is uniform over $(x,a)$ and would otherwise be carried by
$\widetilde E_n$ in the one direction in which $L_\delta$ does not contract. The initial
value $\kappa_0$ is chosen to minimize $\|\widetilde E_0\|_\infty$.

\begin{lemma}[Reduction to a sup-norm contractive stochastic approximation]
\label{lem:reduction}
Let $(Q_n)$ be generated by Algorithm~\ref{alg:softQL} with $f\equiv0$, and let $Q^*:=Q^*_\delta$. Define $\kappa_0:=\tfrac12\bigl[\max_{(x,a)}(Q_0-Q^*)(x,a)+\min_{(x,a)}(Q_0-Q^*)(x,a)\bigr]$, $\widetilde E_0:=Q_0-Q^*-\kappa_0\1$, and recursively, for $n\ge0$,
\begin{equation}\label{eq:kappa-rec}
\kappa_{n+1}:=\kappa_n+\alpha_n\bigl(\rho^*_\delta+c(Q^*+\widetilde E_n)-c(Q^*)\bigr),
\qquad
\widetilde E_{n+1}:=Q_{n+1}-Q^*-\kappa_{n+1}\1 .
\end{equation}
Let $G(E):=\widehat L_\delta(Q^*+E)-\widehat L_\delta Q^*$ for $E\in\R^{S\times A}$, and let $M_{n+1}$ be as in~\eqref{eq:softQL-operator-form}. Then:
\begin{enumerate}[label=(\alph*)]
\item for every $n\ge0$,
\begin{equation}\label{eq:SA-form}
\widetilde E_{n+1}=\widetilde E_n+\alpha_n\bigl(G(\widetilde E_n)-\widetilde E_n+M_{n+1}\bigr);
\end{equation}
\item $G(0)=0$ and $\|G(E)-G(E')\|_\infty\le(1-\varepsilon)\|E-E'\|_\infty$ for all $E,E'$;
\item $\widetilde E_n$ and $\kappa_n$ are $\calF_n$-measurable, $\E[M_{n+1}\mid\calF_n]=0$, and $\|M_{n+1}\|_\infty\le\spn(Q^*)+2\|\widetilde E_n\|_\infty$; in particular
\[
\E\bigl[\|M_{n+1}\|_\infty^2\,\big|\,\calF_n\bigr]\le 2\,\spn(Q^*)^2+8\,\|\widetilde E_n\|_\infty^2 ;
\]
\item $\spn(Q_n-Q^*)=\spn(\widetilde E_n)\le2\|\widetilde E_n\|_\infty$ for all $n\ge0$, and $\|\widetilde E_0\|_\infty=\tfrac12\spn(Q_0-Q^*)$.
\end{enumerate}
\end{lemma}

\begin{proof}{Proof.}
(a) By~\eqref{eq:softQL-operator-form} with $f\equiv0$ and the soft Q-AROE $L_\delta Q^*=Q^*+\rho^*_\delta\1$,
\[
Q_{n+1}-Q^*=(1-\alpha_n)(Q_n-Q^*)+\alpha_n\bigl[L_\delta Q_n-L_\delta Q^*+\rho^*_\delta\1+M_{n+1}\bigr].
\]
Write $Q_n=Q^*+\widetilde E_n+\kappa_n\1$. Using $L_\delta(Q+s\1)=L_\delta Q+s\1$ and the decomposition $L_\delta=\widehat L_\delta+c(\cdot)\1$,
\[
L_\delta Q_n-L_\delta Q^*
=L_\delta(Q^*+\widetilde E_n)-L_\delta Q^*+\kappa_n\1
=G(\widetilde E_n)+\bigl[c(Q^*+\widetilde E_n)-c(Q^*)+\kappa_n\bigr]\1 .
\]
Substituting and collecting the multiples of $\1$,
\[
Q_{n+1}-Q^*
=\Bigl[(1-\alpha_n)\widetilde E_n+\alpha_n G(\widetilde E_n)+\alpha_n M_{n+1}\Bigr]
+\Bigl[\kappa_n+\alpha_n\bigl(\rho^*_\delta+c(Q^*+\widetilde E_n)-c(Q^*)\bigr)\Bigr]\1 ,
\]
and the second bracket is $\kappa_{n+1}$ by~\eqref{eq:kappa-rec}. Hence $\widetilde E_{n+1}=Q_{n+1}-Q^*-\kappa_{n+1}\1=(1-\alpha_n)\widetilde E_n+\alpha_n G(\widetilde E_n)+\alpha_n M_{n+1}$, which is~\eqref{eq:SA-form}.

(b) $G(0)=0$ is immediate, and $G(E)-G(E')=\widehat L_\delta(Q^*+E)-\widehat L_\delta(Q^*+E')$, so the claim is the $(1-\varepsilon)$-contraction of $\widehat L_\delta$ noted before the lemma.

(c) Measurability follows by induction from~\eqref{eq:kappa-rec}, since $Q_n$ is $\calF_n$-measurable. For fixed $(x,a)$, $M_{n+1}(x,a)$ is a function of $Q_n$ and of $Y_n^{x,a}$, and $Y_n^{x,a}\sim P(\cdot\mid x,a)$ is independent of $\calF_n$; hence $\E[M_{n+1}(x,a)\mid\calF_n]=\sum_yP(y\mid x,a)V_{Q_n}(y)-\sum_yP(y\mid x,a)V_{Q_n}(y)=0$. Next, both $V_{Q_n}(Y_n^{x,a})$ and $\sum_yP(y\mid x,a)V_{Q_n}(y)$ lie in the interval $[\min_yV_{Q_n}(y),\max_yV_{Q_n}(y)]$, so $|M_{n+1}(x,a)|\le\spn(V_{Q_n})$. Because the action set is common to all states, Lemma~\ref{lem:lse_lip} gives $|V_{Q_n}(y_1)-V_{Q_n}(y_2)|\le\max_b|Q_n(y_1,b)-Q_n(y_2,b)|\le\spn(Q_n)$ for all $y_1,y_2$, i.e.\ $\spn(V_{Q_n})\le\spn(Q_n)$. Finally, $\spn(Q_n)=\spn(Q^*+\widetilde E_n+\kappa_n\1)\le\spn(Q^*)+\spn(\widetilde E_n)\le\spn(Q^*)+2\|\widetilde E_n\|_\infty$. The bound on the conditional second moment follows from $(u+v)^2\le2u^2+2v^2$.

(d) $\spn(Q_n-Q^*)=\spn(\widetilde E_n+\kappa_n\1)=\spn(\widetilde E_n)\le2\|\widetilde E_n\|_\infty$. For $n=0$, $\|\widetilde E_0\|_\infty=\min_{c\in\R}\|Q_0-Q^*-c\1\|_\infty=\tfrac12\spn(Q_0-Q^*)$ by the choice of $\kappa_0$.\Halmos
\end{proof}

Recursion~\eqref{eq:SA-form} is a stochastic approximation for the fixed point $E=0$ of the sup-norm contraction $G$, driven by martingale-difference noise whose conditional second moment grows at most affinely in $\|\widetilde E_n\|_\infty^2$. This is exactly the setting of \citet{ChMaShSh20}, whose result we now recall in the form needed here.%

\begin{theorem}[\citet{ChMaShSh20}, Theorem~2.1 and Corollaries~2.1.2 and~2.1.3]
\label{thm:chen}
Let $d\ge2$, $\gamma\in(0,1)$, and let $G:\R^d\to\R^d$ satisfy $\|G(x)-G(y)\|_\infty\le\gamma\|x-y\|_\infty$ for all $x,y$, with fixed point $x^*$. Let $(x_n)$ satisfy $x_{n+1}=x_n+\alpha_n\bigl(G(x_n)-x_n+w_n\bigr)$, where $(x_n)$ is adapted to a filtration $(\calF_n)$, $x_0$ is $\calF_0$-measurable with $\E\|x_0-x^*\|_\infty^2<\infty$, $w_n$ is $\calF_{n+1}$-measurable, $\E[w_n\mid\calF_n]=0$ and $\E[\|w_n\|_\infty^2\mid\calF_n]\le A+B\|x_n\|_\infty^2$ for constants $A\ge0$, $B>0$. Let
\[
\alpha_n=\frac{a}{n+n_0},\qquad a:=\frac{4}{1-\gamma},\qquad n_0\ge\frac{256\,e\,(B+2)\log d}{(1-\gamma)^3}.
\]
Then for all $n\ge0$,
\begin{equation}\label{eq:chen-bound}
\E\bigl[\|x_n-x^*\|_\infty^2\bigr]
\le\sqrt{e}\;\E\|x_0-x^*\|_\infty^2\Bigl(\frac{n_0}{n+n_0}\Bigr)^2
+\frac{1024\,e^2\log d}{(1-\gamma)^3}\cdot\frac{A+2B\|x^*\|_\infty^2}{n+n_0}.
\end{equation}
\end{theorem}

\begin{proof}{Derivation from \citet{ChMaShSh20}.}
In the notation of that paper, with $\|\cdot\|_c=\|\cdot\|_e=\|\cdot\|_\infty$, the smoothing norm $\|\cdot\|_p$ with $p=4\log d$ and $\mu=(1/2+1/(2\gamma))^2-1$, Corollary~2.1.3 there gives $\alpha_2\ge(1-\gamma)/2$, $\alpha_3\le32e(B+2)\log d/(1-\gamma)$ and $\alpha_4\le16e\log d/(1-\gamma)$. For $\alpha_1=(1+\mu/\ell_{cs}^2)/(1+\mu/u_{cs}^2)$ we use the norm-equivalence constants $\ell_{cs}=1$ and $u_{cs}=d^{1/p}=e^{1/4}$ of the pair $(\|\cdot\|_\infty,\|\cdot\|_p)$, which give $\alpha_1=(1+\mu)/(1+\mu e^{-1/2})<\sqrt e$ for every $\mu>0$. Our $a=4/(1-\gamma)$ satisfies $a\ge2/\alpha_2>1/\alpha_2$, so case~3 of Corollary~2.1.2 applies with $K=\max(1,a\alpha_3/\alpha_2)\le256e(B+2)\log d/(1-\gamma)^3\le n_0$, and yields
\[
\E\|x_n-x^*\|_\infty^2\le\alpha_1\,\E\|x_0-x^*\|_\infty^2\Bigl(\frac{n_0}{n+n_0}\Bigr)^{\alpha_2a}+\frac{4e\,a^2\alpha_4}{\alpha_2a-1}\cdot\frac{A+2B\|x^*\|_\infty^2}{n+n_0}.
\]
Since $\alpha_2a\ge2$, the first term is at most $\sqrt e\,\E\|x_0-x^*\|^2_\infty(n_0/(n+n_0))^2$, and since $\alpha_2a-1\ge1$, $a^2=16/(1-\gamma)^2$ and $\alpha_4\le16e\log d/(1-\gamma)$, the coefficient of the second term is at most $1024e^2\log d/(1-\gamma)^3$. \Halmos

We note that \citet[Corollary~2.1.2]{ChMaShSh20} is stated with $K=\max(1,\epsilon\alpha_3/\alpha_2)$ in their notation; as they point out, the only role of $K$ is to guarantee that the initial step size satisfies $\epsilon_0\le\alpha_2/\alpha_3$, and their proof applies unchanged to any larger $K$. Moreover, their filtration $\{x_0,w_0,\dots,x_k\}$ may be replaced by any filtration to which $(x_k)$ is adapted and for which the two noise conditions hold.
\end{proof}

We also need a uniform bound on the span of $Q^*_\delta$, which is the single-agent version of Lemma~\ref{lem:span-Qm}.

\begin{lemma}[Span of the soft Q-function]
\label{lem:span-Qstar}
Under Assumption~\ref{def:minor}, $\spn(Q^*_\delta)\le\spn(r)/\varepsilon$, where $\spn(r):=\max_{(x,a)}r(x,a)-\min_{(x,a)}r(x,a)$.
\end{lemma}

\begin{proof}{Proof.}
The proof is similar to the proof of Lemma~\ref{lem:span-Qm}.\Halmos
\end{proof}

\begin{theorem}[Finite-sample bound for synchronous soft Q-learning]
\label{thm:softQL-rate}
Let Assumption~\ref{def:minor} hold, let $d=|S||A|\ge2$, and let $(Q_n)$ be generated by Algorithm~\ref{alg:softQL} with an arbitrary offset $f$, an initial condition $Q_0$ (deterministic, or random with finite second moment and independent of the samples), and step sizes
\begin{equation}\label{eq:stepsizes}
\alpha_n=\frac{a}{n+n_0},\qquad a:=\frac4\varepsilon,\qquad n_0\ge\frac{2560\,e\log d}{\varepsilon^3}.
\end{equation}
Then for every $n\ge0$,
\begin{equation}\label{eq:softQL-bound}
\E\bigl[\spn(Q_n-Q^*_\delta)^2\bigr]
\le\sqrt e\;\E\bigl[\spn(Q_0-Q^*_\delta)^2\bigr]\Bigl(\frac{n_0}{n+n_0}\Bigr)^2
+\frac{8192\,e^2\log d}{\varepsilon^3}\cdot\frac{\spn(Q^*_\delta)^2}{n+n_0},
\end{equation}
where $\spn(Q^*_\delta)\le\spn(r)/\varepsilon$. Consequently, the policy $\pi_n=\sigma_\delta(Q_n)$ satisfies, for every $x\in S$,
\begin{equation}\label{eq:softQL-policy-bound}
\E\bigl[\|\pi_n(\cdot\mid x)-\pi^*_\delta(\cdot\mid x)\|_1^2\bigr]\le\frac{1}{4\delta^2}\,\E\bigl[\spn(Q_n-Q^*_\delta)^2\bigr],
\end{equation}
where $\pi^*_\delta=\sigma_\delta(Q^*_\delta)$ is the optimal entropy-regularized policy.
\end{theorem}

\begin{proof}{Proof.}
By Lemma~\ref{lem:offset-invariance}, both sides of~\eqref{eq:softQL-bound} and~\eqref{eq:softQL-policy-bound} are unchanged if $f$ is replaced by $f\equiv0$, so we assume $f\equiv0$. By Lemma~\ref{lem:reduction}, the sequence $(\widetilde E_n)$ satisfies the recursion of Theorem~\ref{thm:chen} with $\gamma=1-\varepsilon$, $x^*=0$, $w_n=M_{n+1}$, $A=2\spn(Q^*_\delta)^2$ and $B=8$; the step sizes~\eqref{eq:stepsizes} are those of Theorem~\ref{thm:chen}, since $4/(1-\gamma)=4/\varepsilon$ and $256e(B+2)\log d/(1-\gamma)^3=2560e\log d/\varepsilon^3$. Hence, by~\eqref{eq:chen-bound},
\[
\E\|\widetilde E_n\|_\infty^2\le\sqrt e\;\E\|\widetilde E_0\|_\infty^2\Bigl(\frac{n_0}{n+n_0}\Bigr)^2+\frac{1024\,e^2\log d}{\varepsilon^3}\cdot\frac{2\,\spn(Q^*_\delta)^2}{n+n_0}.
\]
By Lemma~\ref{lem:reduction}(d), $\spn(Q_n-Q^*_\delta)^2\le4\|\widetilde E_n\|_\infty^2$ and $\|\widetilde E_0\|_\infty^2=\frac14\spn(Q_0-Q^*_\delta)^2$; multiplying the last display by $4$ gives~\eqref{eq:softQL-bound}. The bound on $\spn(Q^*_\delta)$ is Lemma~\ref{lem:span-Qstar}. Finally, by the shift-invariant softmax bound~\eqref{eq:softmax_span_lip},
\begin{equation*}
\|\pi_n(\cdot\mid x)-\pi^*_\delta(\cdot\mid x)\|_1=\|\sigma_\delta(Q_n(x,\cdot))-\sigma_\delta(Q^*_\delta(x,\cdot))\|_1
\le\frac{\spn\bigl(Q_n(x,\cdot)-Q^*_\delta(x,\cdot)\bigr)}{2\delta}\le\frac{\spn(Q_n-Q^*_\delta)}{2\delta},
\end{equation*}
and squaring and taking expectations gives~\eqref{eq:softQL-policy-bound}.\Halmos
\end{proof}

\begin{remark}[Discussion of the bound]
\label{rem:softQL-discussion}
(i) The mean-square error of $Q_n$ in the span seminorm decays as $O(1/n)$; the constants depend on the size of the problem only through $\log(|S||A|)$, and on the minorization mass through $\varepsilon^{-3}$, in complete analogy with the dependence on $(1-\beta)^{-3}$ in the discounted case \citep{ChMaShSh20}. The temperature $\delta$ does not appear in~\eqref{eq:softQL-bound} at all: as in Lemma~\ref{lem:lipschitz-Q}, this is because the log-sum-exp is $1$-Lipschitz for every $\delta$. The temperature enters only through the softmax in~\eqref{eq:softQL-policy-bound}, at the rate $1/\delta^2$. %

(ii) Finite-sample bounds for hard-max RVI Q-learning in the span seminorm were obtained by \citet{ZhZhMa21} under a span-contraction assumption on a multi-step Bellman operator, via a Lyapunov argument in the quotient space $\R^{S\times A}/\R\1$; the general theory of \citet{ChZhZhHaMa25} for seminorm-contractive stochastic approximation applies to $L_\delta$ as well, since $L_\delta$ is a span-seminorm contraction by Theorem~\ref{thm:Lhat_delta_span}. The route taken here is more elementary and makes the role of the minorization condition transparent: the measure $\nu$, which the learner does not know, is used only to define the scalar sequence $(\kappa_n)$ in Lemma~\ref{lem:reduction}, after which the analysis is that of an ordinary sup-norm contractive stochastic approximation.
\end{remark}

\subsection{Learning the mean-field equilibrium}
\label{sec:mfg-learning}

We now return to the mean-field game of Section~\ref{sec:mfg}, under Assumptions~\ref{ass:lip_primitives} and~\ref{ass:minor_mfg}, and assume that $K_1<\varepsilon$ and that $\delta$ satisfies~\eqref{eq:delta-large}, so that $\calF$ is a $C_{\calF}$-contraction with $C_{\calF}<1$ by Theorem~\ref{thm:F-contraction}; $m^*$ denotes the unique equilibrium population distribution and $\pi^*:=\pi^*_{m^*}$ the canonical equilibrium policy of Corollary~\ref{cor:MFNE}. The learner has access to the simulator described at the beginning of the section. The model-free version of the iteration of Section~\ref{sec:algorithm} is Algorithm~\ref{alg:mfg-learning}.

\begin{algorithm}[t]
\caption{Model-free mean-field iteration with soft Q-learning and empirical kernels}
\label{alg:mfg-learning}
\begin{algorithmic}[1]
\Require $m_0\in\Delta(S)$; temperature $\delta$; inner iteration numbers $(n_k)_{k\ge0}$ and $(N_k)_{k\ge0}$; step sizes $(\alpha_n)$ as in~\eqref{eq:stepsizes}
\For{$k=0,1,2,\dots$}
  \State \textbf{(a) soft Q-learning at population $m_k$:} run Algorithm~\ref{alg:softQL} for $n_k$ iterations, querying the simulator at $(x,a,m_k)$, from $Q_0^{(k)}:=0$; set $\widehat Q_k:=Q_{n_k}$ and $\widehat\pi_k:=\sigma_\delta(\widehat Q_k)$
  \State \textbf{(b) empirical kernel:} for each $(x,a)$ draw $N_k$ fresh samples $Y^{x,a}_{k,i}\sim P(\cdot| x,a,m_k)$, $i=1,\dots,N_k$, and set
  $\widehat P_k(y| x,a):=\frac1{N_k}\sum_{i=1}^{N_k}\1_{\{Y^{x,a}_{k,i}=y\}}$
  \State \textbf{(c) population update:} $m_{k+1}(y):=\sum_{x\in S}\sum_{a\in A}\widehat P_k(y| x,a)\,\widehat\pi_k(a| x)\,m_k(x)$, $y\in S$
\EndFor
\Ensure $(\widehat\pi_k,m_k)$
\end{algorithmic}
\end{algorithm}

Since every row $\widehat P_k(\cdot\mid x,a)$ is a probability vector, $m_{k+1}\in\Delta(S)$ automatically. Each outer iteration uses $|S||A|(n_k+N_k)$ simulator calls. Step~(c) is the exact population update of Definition~\ref{def:mfe_op} with $(P,\pi^*_{m_k})$ replaced by $(\widehat P_k,\widehat\pi_k)$, and the analysis consists in controlling the two resulting errors and propagating them through the contraction $\calF$. We first record the accuracy of empirical distributions.

\begin{lemma}[Empirical distributions]
\label{lem:empirical}
Let $p\in\Delta(S)$ and let $\widehat p_N$ be the empirical distribution of $N$ i.i.d.\ samples from $p$. Then $\E\|\widehat p_N-p\|_1\le\sqrt{|S|/N}$ and $\E\|\widehat p_N-p\|_1^2\le|S|/N$.
\end{lemma}

\begin{proof}{Proof.}
For each $y\in S$, $N\widehat p_N(y)$ is binomial with parameters $(N,p(y))$, so by the Cauchy--Schwarz inequality $\E|\widehat p_N(y)-p(y)|\le\sqrt{\operatorname{Var}(\widehat p_N(y))}=\sqrt{p(y)(1-p(y))/N}\le\sqrt{p(y)/N}$. Summing over $y$ and using $\sum_y\sqrt{p(y)}\le\sqrt{|S|\sum_yp(y)}=\sqrt{|S|}$ (Cauchy--Schwarz again) gives the first claim. For the second, $\|\widehat p_N-p\|_1^2\le|S|\sum_y(\widehat p_N(y)-p(y))^2$ by the Cauchy--Schwarz inequality, and $\E\sum_y(\widehat p_N(y)-p(y))^2=\sum_yp(y)(1-p(y))/N\le1/N$.\Halmos
\end{proof}

\begin{lemma}[Error of the approximate mean-field operator]
\label{lem:approx-F}
Let Assumptions~\ref{ass:lip_primitives} and~\ref{ass:minor_mfg} hold, so that $Q^*_m$, $\pi^*_m$ and $\calF$ are well defined. Let $m\in\Delta(S)$, $\widehat Q\in\R^{S\times A}$, $\widehat\pi:=\sigma_\delta(\widehat Q)$, and let $\widehat P(\cdot\mid x,a)\in\Delta(S)$ for every $(x,a)$. Define $\widehat\calF(m)(y):=\sum_{x,a}\widehat P(y\mid x,a)\,\widehat\pi(a\mid x)\,m(x)$. Then
\[
\bigl\|\widehat\calF(m)-\calF(m)\bigr\|_1
\le\sum_{x\in S}\sum_{a\in A}m(x)\,\widehat\pi(a\mid x)\,\bigl\|\widehat P(\cdot\mid x,a)-P(\cdot\mid x,a,m)\bigr\|_1
+\frac{K_1}{2\delta}\,\spn\bigl(\widehat Q-Q^*_m\bigr).
\]
\end{lemma}

\begin{proof}{Proof.}
By the definitions of $\widehat\calF$ and $\calF$,
\begin{multline*}
\widehat\calF(m)(y)-\calF(m)(y)
=\sum_{x,a}\bigl(\widehat P(y\mid x,a)-P(y\mid x,a,m)\bigr)\widehat\pi(a\mid x)m(x)\\
\quad+\sum_x m(x)\sum_a\bigl(\widehat\pi(a\mid x)-\pi^*_m(a\mid x)\bigr)P(y\mid x,a,m).
\end{multline*}
The $\ell^1$-norm of the first sum is at most $\sum_{x,a}m(x)\widehat\pi(a\mid x)\|\widehat P(\cdot\mid x,a)-P(\cdot\mid x,a,m)\|_1$ by the triangle inequality. For the second sum, fix $x$; the weights $\varphi(a):=\widehat\pi(a\mid x)-\pi^*_m(a\mid x)$ sum to zero, so, exactly as in the policy part of the proof of Proposition~\ref{prop:policy-avg-lip}, Lemma~\ref{lem:span-tv} and the action part of~\eqref{eq:lip_kernel} give
\begin{align*}
\Bigl\|\sum_a\varphi(a)P(\cdot\mid x,a,m)\Bigr\|_1 \le K_1\,\|\widehat\pi(\cdot\mid x)-\pi^*_m(\cdot\mid x)\|_1
\le\frac{K_1}{2\delta}\,\spn\bigl(\widehat Q(x,\cdot)-Q^*_m(x,\cdot)\bigr)\le\frac{K_1}{2\delta}\,\spn\bigl(\widehat Q-Q^*_m\bigr),
\end{align*}
where the second inequality is the shift-invariant softmax bound~\eqref{eq:softmax_span_lip} applied to $\widehat\pi(\cdot\mid x)=\sigma_\delta(\widehat Q(x,\cdot))$ and $\pi^*_m(\cdot\mid x)=\sigma_\delta(Q^*_m(x,\cdot))$. Multiplying by $m(x)$, summing over $x$ and combining the two bounds proves the claim.\Halmos
\end{proof}

\begin{theorem}[Convergence of the model-free mean-field iteration]
\label{thm:mfg-learning}
Let Assumptions~\ref{ass:lip_primitives} and~\ref{ass:minor_mfg} hold, let $K_1<\varepsilon$ and let $\delta$ satisfy~\eqref{eq:delta-large}, so that $C_{\calF}<1$. Let $d=|S||A|\ge2$ and let the inner step sizes be as in~\eqref{eq:stepsizes}. Define, for $n\ge0$ and $k\ge0$,
\begin{equation}\label{eq:beta-eta}
\beta(n):=\Bigl(\frac{2L_1}{\varepsilon}\Bigr)^2\left[\sqrt e\Bigl(\frac{n_0}{n+n_0}\Bigr)^2+\frac{8192\,e^2\log d}{\varepsilon^3\,(n+n_0)}\right],
\qquad
\bar\eta_k:=\frac{K_1}{2\delta}\sqrt{\beta(n_k)}+\sqrt{\frac{|S|}{N_k}} .
\end{equation}
Then the iterates of Algorithm~\ref{alg:mfg-learning} satisfy, for every $k\ge0$,
\begin{equation}\label{eq:mfg-learning-bound}
\E\|m_k-m^*\|_1\le C_{\calF}^{\,k}\,\|m_0-m^*\|_1+\sum_{j=0}^{k-1}C_{\calF}^{\,k-1-j}\,\bar\eta_j ,
\end{equation}
and the learned policies satisfy, for every $k\ge0$ and $x\in S$,
\begin{equation}\label{eq:mfg-learning-policy}
\E\|\widehat\pi_k(\cdot\mid x)-\pi^*(\cdot\mid x)\|_1\le\frac{1}{2\delta}\sqrt{\beta(n_k)}+\frac{C_Q}{2\delta}\,\E\|m_k-m^*\|_1 .
\end{equation}
\end{theorem}

\begin{proof}{Proof.}
Let $\calG_k$ be the $\sigma$-field generated by all simulator outputs used in outer iterations $0,\dots,k-1$, so that $m_k$ is $\calG_k$-measurable, and let $\calH_k\supseteq\calG_k$ be the $\sigma$-field generated additionally by the outputs used in step~(a) of iteration $k$, so that $\widehat Q_k$ and $\widehat\pi_k$ are $\calH_k$-measurable.

\smallskip\noindent\emph{Step 1 (inner Q-learning error).}
Conditionally on $\calG_k$, step~(a) is Algorithm~\ref{alg:softQL} applied to the MDP with kernel $P(\cdot\mid\cdot,\cdot,m_k)$ and reward $r(\cdot,\cdot,m_k)$, with fresh samples. Hence Theorem~\ref{thm:softQL-rate} applies conditionally on $\calG_k$. Since $Q_0^{(k)}=0$, we have $\spn(Q_0^{(k)}-Q^*_{m_k})=\spn(Q^*_{m_k})\le2L_1/\varepsilon$ by Lemma~\ref{lem:span-Qm}, and~\eqref{eq:softQL-bound} gives
\( \displaystyle
\E\bigl[\spn(\widehat Q_k-Q^*_{m_k})^2\,\big|\,\calG_k\bigr]\le\beta(n_k)\),
hence
$ \displaystyle \E\bigl[\spn(\widehat Q_k-Q^*_{m_k})\,\big|\,\calG_k\bigr]\le\sqrt{\beta(n_k)}$
by Jensen's inequality.

\smallskip\noindent\emph{Step 2 (kernel estimation error).}
Conditionally on $\calH_k$, the samples of step~(b) are $N_k$ i.i.d.\ draws from $P(\cdot| x,a,m_k)$ for each $(x,a)$, so Lemma~\ref{lem:empirical} gives $\E[\|\widehat P_k(\cdot| x,a)-P(\cdot| x,a,m_k)\|_1|\calH_k]\le\sqrt{|S|/N_k}$. Since $m_k$ and $\widehat\pi_k$ are $\calH_k$-measurable,
\begin{align*}
\E\Bigl[\sum_{x,a}m_k(x)\widehat\pi_k(a\mid x)\bigl\|\widehat P_k(\cdot| x,a)-P(\cdot| x,a,m_k)\bigr\|_1\,\Big|\,\calH_k\Bigr]
\le\sqrt{\frac{|S|}{N_k}}\sum_{x,a}m_k(x)\widehat\pi_k(a| x)
=\sqrt{\frac{|S|}{N_k}} .
\end{align*}

\smallskip\noindent\emph{Step 3 (one-step error and propagation).}
By Lemma~\ref{lem:approx-F} applied with $m=m_k$, $\widehat Q=\widehat Q_k$ and $\widehat P=\widehat P_k$, and by Steps~1 and~2,
\[
\E\|m_{k+1}-\calF(m_k)\|_1\le\sqrt{\frac{|S|}{N_k}}+\frac{K_1}{2\delta}\sqrt{\beta(n_k)}=\bar\eta_k .
\]
By Theorem~\ref{thm:F-contraction} and the triangle inequality,
\begin{align*}
\|m_{k+1}-m^*\|_1&\le\|\calF(m_k)-\calF(m^*)\|_1+\|m_{k+1}-\calF(m_k)\|_1\\
&\le C_{\calF}\|m_k-m^*\|_1+\|m_{k+1}-\calF(m_k)\|_1 .
\end{align*}
Taking expectations, $e_{k+1}\le C_{\calF}e_k+\bar\eta_k$ with $e_k:=\E\|m_k-m^*\|_1$, and induction on $k$ gives~\eqref{eq:mfg-learning-bound}.

\smallskip\noindent\emph{Step 4 (policy).}
By the triangle inequality, the shift-invariant softmax bound~\eqref{eq:softmax_span_lip} and Lemma~\ref{lem:lipschitz-policy},
\begin{align*}
\|\widehat\pi_k(\cdot\mid x)-\pi^*(\cdot\mid x)\|_1
&\le\|\sigma_\delta(\widehat Q_k(x,\cdot))-\sigma_\delta(Q^*_{m_k}(x,\cdot))\|_1+\|\pi^*_{m_k}(\cdot\mid x)-\pi^*_{m^*}(\cdot\mid x)\|_1\\
&\le\frac{\spn(\widehat Q_k-Q^*_{m_k})}{2\delta}+\frac{C_Q}{2\delta}\|m_k-m^*\|_1,
\end{align*}
and taking expectations and using Step~1 gives~\eqref{eq:mfg-learning-policy}.\Halmos
\end{proof}

\begin{corollary}[Sample complexity]
\label{cor:budgets}
Under the assumptions of Theorem~\ref{thm:mfg-learning}:
\begin{enumerate}
\item If $n_k\to\infty$ and $N_k\to\infty$, then $\E\|m_k-m^*\|_1\to0$ and $\E\|\widehat\pi_k(\cdot\mid x)-\pi^*(\cdot\mid x)\|_1\to0$ for every $x\in S$.
\item If $n_k\equiv n$ and $N_k\equiv N$, then for every $k\ge0$ and $x\in S$,
\begin{align*}
\E\|m_k-m^*\|_1&\le C_{\calF}^{\,k}\|m_0-m^*\|_1+\frac{1}{1-C_{\calF}}\left[\frac{K_1}{2\delta}\sqrt{\beta(n)}+\sqrt{\frac{|S|}{N}}\right],\\
\E\|\widehat\pi_k(\cdot\mid x)-\pi^*(\cdot\mid x)\|_1&\le\frac{C_Q}{2\delta}\,C_{\calF}^{\,k}\|m_0-m^*\|_1\\
&\quad+\frac{1}{2\delta}\left[\Bigl(1+\frac{C_QK_1}{2\delta(1-C_{\calF})}\Bigr)\sqrt{\beta(n)}+\frac{C_Q}{1-C_{\calF}}\sqrt{\frac{|S|}{N}}\right],
\end{align*}
and in both bounds the bracket is $O\bigl(n^{-1/2}+N^{-1/2}\bigr)$.
\item Let $\theta\in(C_{\calF},1)$ and choose the number of iterations in the inner loops so that
\(
\sqrt{\beta(n_k)}\le\eta_1\theta^k \) and \(\sqrt{|S|/N_k}\le\eta_2\theta^k\) for all $k\ge0$
and some constants $\eta_1,\eta_2>0$; since $(n_0/(n+n_0))^2\le n_0/(n+n_0)$ in~\eqref{eq:beta-eta}, this holds whenever $n_k+n_0\ge c_1\theta^{-2k}$ and $N_k\ge c_2\theta^{-2k}$, where
\[
c_1:=\frac{1}{\eta_1^2}\Bigl(\frac{2L_1}{\varepsilon}\Bigr)^2\Bigl[\sqrt e\,n_0+\frac{8192\,e^2\log d}{\varepsilon^3}\Bigr],
\qquad
c_2:=\frac{|S|}{\eta_2^2}.
\]
Then, with $\eta:=\frac{K_1}{2\delta}\eta_1+\eta_2$, for every $k\ge0$ and $x\in S$,
\begin{align*}
\E\|m_k-m^*\|_1&\le\Bigl(\|m_0-m^*\|_1+\frac{\eta}{\theta-C_{\calF}}\Bigr)\theta^{k},\\
\E\|\widehat\pi_k(\cdot\mid x)-\pi^*(\cdot\mid x)\|_1&\le\frac{1}{2\delta}\Bigl[\eta_1+C_Q\Bigl(\|m_0-m^*\|_1+\frac{\eta}{\theta-C_{\calF}}\Bigr)\Bigr]\theta^{k}.
\end{align*}
\end{enumerate}
\end{corollary}

\begin{proof}{Proof.}
(1) By~\eqref{eq:beta-eta}, $\bar\eta_k\to0$. For $J<k$,
\[
\sum_{j=0}^{k-1}C_{\calF}^{k-1-j}\bar\eta_j
\le C_{\calF}^{\,k-J}\sum_{j=0}^{J-1}C_{\calF}^{J-1-j}\bar\eta_j+\frac{\sup_{j\ge J}\bar\eta_j}{1-C_{\calF}} ;
\]
letting first $k\to\infty$ and then $J\to\infty$ shows that the sum in~\eqref{eq:mfg-learning-bound} tends to zero, and the policy statement follows from~\eqref{eq:mfg-learning-policy}. (2) Bound the geometric sum in~\eqref{eq:mfg-learning-bound} by $1/(1-C_{\calF})$, which gives the first bound, and insert it into~\eqref{eq:mfg-learning-policy}, which gives the second; the order statement follows from $(n_0/(n+n_0))^2\le n_0/(n+n_0)$ in~\eqref{eq:beta-eta}. (3) The growth conditions give $\beta(n_k)\le c_1\eta_1^2/(n_k+n_0)\le\eta_1^2\theta^{2k}$ and $|S|/N_k\le\eta_2^2\theta^{2k}$, hence $\bar\eta_k\le\eta\,\theta^k$. Then
\[
\sum_{j=0}^{k-1}C_{\calF}^{k-1-j}\bar\eta_j\le\eta\,\theta^{k-1}\sum_{j=0}^{k-1}\Bigl(\frac{C_{\calF}}{\theta}\Bigr)^{k-1-j}\le\frac{\eta\,\theta^k}{\theta-C_{\calF}} ,
\]
which together with~\eqref{eq:mfg-learning-bound} gives the bound on $\E\|m_k-m^*\|_1$; inserting it and $\sqrt{\beta(n_k)}\le\eta_1\theta^k$ into~\eqref{eq:mfg-learning-policy} gives the bound on the policy.\Halmos
\end{proof}

\section{Numerical Experiments}
\label{sec:numerics}

We illustrate the theory on a congestion-aware routing game. We first compute the exact mean-field equilibrium by the Banach fixed-point iteration. Then we run the model-free Algorithms~\ref{alg:softQL} and~\ref{alg:mfg-learning} on the same instance and compare their behavior with the rates of Section~\ref{sec:learning}.

\subsection{The model as a network-congestion game}
\label{sub:val-network}

The scenario is a stylized communication network: a large number of users repeatedly establish a connection to a server, each
choosing between two routes. The \emph{stable} route is reliable but expensive;
the \emph{cheap} route is inexpensive but its quality degrades as more users
load the network. Formally,
\[
  S = \{G, B\} \quad (\text{good / bad connection}),
  \qquad
  A = \{\mathsf{S}, \mathsf{C}\} \quad (\text{Stable / Cheap route}),
\]
and the population coordinate $m_B$ is the \emph{congestion level}, i.e.\ the
fraction of users currently experiencing a bad connection. The mean-field
coupling is genuine: both the transition kernel and the reward depend on $m_B$,
so each user's prospects are shaped by the aggregate behavior of all the others.
The transition kernel and one-stage reward are
\begin{equation}
  P(G \mid x, a, m) = g_{x,a} - \lambda\, m_B,
  \qquad
  r(x, a, m) = w\,\mathbf 1\{x = G\} - \mathrm{price}_a - \kappa\, m_B,
  \label{eq:val-primitives}
\end{equation}
with the route-quality table $g$, the congestion slope $\lambda$, the
connection value $w$, the route prices, the congestion penalty $\kappa$, and the
entropy temperature $\delta$ collected in Table~\ref{tab:val-params}. The
derived constants of Assumptions~\ref{ass:lip_primitives}--\ref{ass:minor_mfg}
and the contraction threshold~\eqref{eq:delta-large}
are also listed there; in particular $K_1 = 0.20 < \varepsilon =
0.68$ and $\delta = 2.0 > \delta^\star \approx 0.793$, so all hypotheses of
Theorem~\ref{thm:F-contraction} hold and $\mathcal{F}$ is a strict
contraction with modulus $C_{\mathcal{F}} \approx 0.71$.

\begin{table}[t]
  \centering
  \caption{Parameters of the congestion-aware routing instance. In the route-quality table $g$, rows are indexed by the state $x\in\{G,B\}$ and columns by the action $a\in\{\mathsf S,\mathsf C\}$. The derived constants satisfy $K_1 < \varepsilon$ and $\delta > \delta^\star$, so the mean-field operator is a contraction. \label{tab:val-params}}
  \begin{tabular}{lll}
    \toprule
    \multicolumn{3}{l}{\textbf{Model primitives}}\\
    \midrule
    Route-quality table & $g$ &
      $\left[\begin{smallmatrix} 0.85 & 0.78\\ 0.82 & 0.68 \end{smallmatrix}\right]$\\
    Congestion slope        & $\lambda$ & $0.15$\\
    Connection value        & $w$ & $1.00$\\
    Route prices            & $(\mathrm{price}_{\mathsf S},\mathrm{price}_{\mathsf C})$ & $(0.50,\,0.20)$\\
    Congestion penalty      & $\kappa$ & $0.40$\\
    Entropy temperature     & $\delta$ & $2.00$\\
    \midrule
    \multicolumn{3}{l}{\textbf{Derived constants}}\\
    \midrule
    Kernel Lipschitz constant   & $K_1$ & $0.20$\\
    Reward Lipschitz constant   & $L_1$ & $1.00$\\
    Minorization mass           & $\varepsilon$ & $0.68$\\
    Temperature threshold       & $\delta^\star$ & $0.793$\\
    Contraction modulus         & $C_{\mathcal{F}}$ & $0.71$\\
    \bottomrule
  \end{tabular}
\end{table}

\subsection{Mean-field equilibrium}

We run the Banach iteration $m^{(k+1)} = \mathcal{F}(m^{(k)})$ of
Corollary~\ref{cor:MFNE} from $m^{(0)} = (\tfrac12,\tfrac12)$.
Each step solves the soft Q-AROE \eqref{eq:soft_Q_AROE} at the current
population by relative value iteration on $L_\delta$ (Corollary~\ref{cor:Lhat_soft_banach}) and forms the softmax policy
\eqref{eq:mfg_softmax} in closed form. This procedure uses no
randomness whatsoever and converges to the exact equilibrium $m^*$ (to numerical
tolerance) in $20$ outer iterations. The results can be found in Table~\ref{tab:val-compare}.

\begin{table}[ht]
  \centering
  \caption{Equilibrium computed by the deterministic mean-field method. The associated optimal entropy-regularized average reward is $\rho^*_\delta = 1.7076$, and the equilibrium policy is $\pi^*(\cdot\mid G) = (\mathsf S\!:\nobreak\!0.472,\,\mathsf C\!:\nobreak\!0.528)$, $\pi^*(\cdot\mid B) = (\mathsf S\!:\nobreak\!0.481,\,\mathsf C\!:\nobreak\!0.519)$. \label{tab:val-compare}}
  \begin{tabular}{lc}
    \toprule
    component &
      \makecell{mean-field term \\ $m^*$}\\
    \midrule
    $m_G$ (good connection) & $0.761626$ \\
    $m_B$ (congestion)      & $0.238374$ \\
    \bottomrule
  \end{tabular}
\end{table}

\subsection{Model-free learning}
\label{sub:learning-exp}

We now run the model-free algorithms of Section~\ref{sec:learning} on the same instance, using a simulator that, given $(x,a,m)$, returns $r(x,a,m)$ and a sample from $P(\cdot\mid x,a,m)$ in~\eqref{eq:val-primitives}. All expectations below are estimated by averaging over independent runs ($200$ runs in panel~(a) of Figure~\ref{fig:learning}, $100$ runs in panel~(b), $50$ runs in panel~(c)). The temperature is $\delta=2$ as before, so that Theorems~\ref{thm:softQL-rate} and~\ref{thm:mfg-learning} apply with $\varepsilon=0.68$, $K_1=0.2$, $L_1=1$ and $d=|S||A|=4$; the step-size constant is $a=4/\varepsilon\approx5.88$ and the offset is $f(Q)=Q(G,\mathsf S)$.

\paragraph{Soft Q-learning at a fixed population.}
Panel~(a) of Figure~\ref{fig:learning} shows the mean-square span error $\E[\spn(Q_n-Q^*_{\delta})^2]$ of Algorithm~\ref{alg:softQL} run at the equilibrium population $m^*$, where $Q^*_\delta=Q^*_{m^*}$ is computed by relative value iteration. Two step-size sequences are used: $\alpha_n=a/(n+n_0)$ with the value $n_0=\lceil2560\,e\log d/\varepsilon^3\rceil=30{,}681$ required by Theorem~\ref{thm:softQL-rate}, and the same sequence with $n_0=100$. In both cases the error eventually decays at the rate $1/n$ predicted by~\eqref{eq:softQL-bound}: the least-squares slopes of the two curves are $-0.88$ and $-0.93$.

\paragraph{Model-free mean-field iteration.}
Panel~(b) shows $\E\|m_k-m^*\|_1$ for Algorithm~\ref{alg:mfg-learning} started from $m_0=(\tfrac12,\tfrac12)$ with a constant number of inner iterations $n_k=N_k\in\{10^2,10^3,10^4\}$ and $n_0=100$, together with the error of the exact (model-based) iteration of Corollary~\ref{cor:MFNE}. As predicted by Corollary~\ref{cor:budgets}(2), the model-free iteration first tracks the geometric decay of the exact iteration and then settles at a floor whose height decreases with the number of iterations: the floors are approximately $4\times10^{-2}$, $1.2\times10^{-2}$ and $3.7\times10^{-3}$ for the three values, i.e.\ they shrink by a factor of about $\sqrt{10}$ when the number of iterations is multiplied by $10$, which is the $O(n^{-1/2}+N^{-1/2})$ dependence of~\eqref{eq:beta-eta}. Panel~(c) uses geometrically growing numbers of inner iterations $n_k=N_k=\lceil b\,\theta^{-2k}\rceil$ with $\theta=0.8>C_{\calF}$ and $b\in\{50,200\}$, and plots the error against the outer iteration $k$. The error then decreases without a floor, geometrically in $k$, at a rate consistent with the bound $\theta^k$ of Corollary~\ref{cor:budgets}(3) (dotted line).

\begin{figure}[t]
  \centering
  \includegraphics[width=\textwidth]{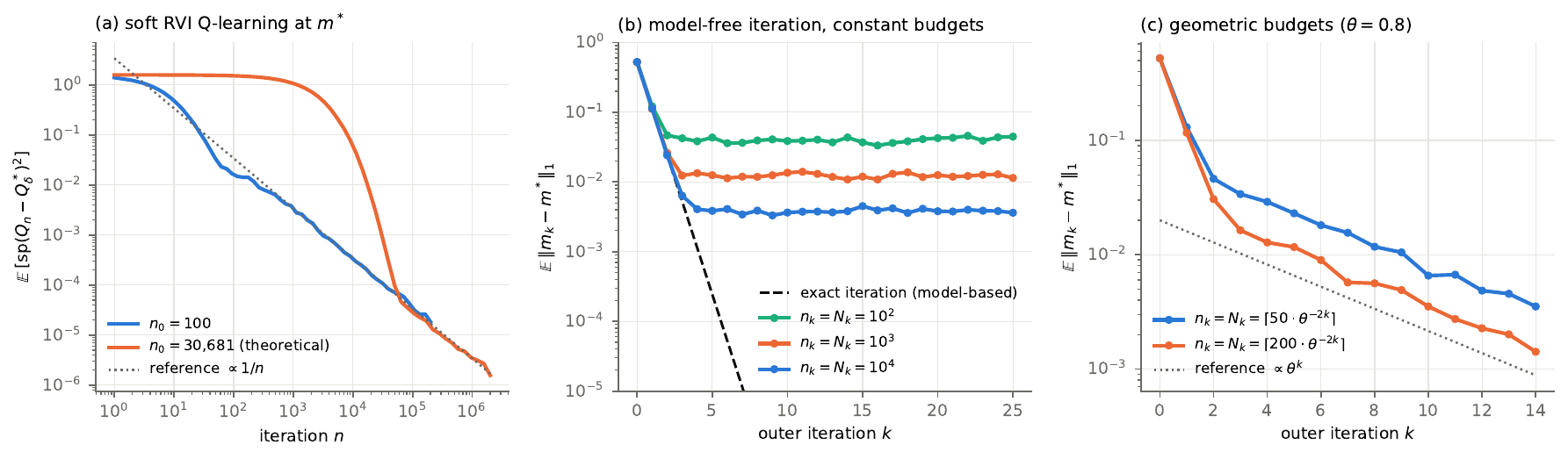}
  \caption{Model-free learning on the congestion-aware routing instance ($\delta=2$). (a) Mean-square span error of synchronous soft RVI Q-learning (Algorithm~\ref{alg:softQL}) at the equilibrium population, for the step sizes $\alpha_n=a/(n+n_0)$ with $a=4/\varepsilon$ and two values of $n_0$; the dotted line is proportional to $1/n$. (b) Model-free mean-field iteration (Algorithm~\ref{alg:mfg-learning}) with a constant number of inner iterations, compared with the exact iteration; the error settles at a floor of order $n^{-1/2}+N^{-1/2}$. (c) The same algorithm with geometrically growing budgets $n_k=N_k=\lceil b\,\theta^{-2k}\rceil$, $\theta=0.8$; the dotted line is proportional to $\theta^{k}$. \label{fig:learning}}
\end{figure}

\section{Conclusion}
\label{sec:conclusion}

This paper studied discrete-time mean-field games under the long-run average-reward criterion with an entropy-regularized objective. The unifying idea was a single substitution: the hard maximization over actions in the average-reward optimality equation was replaced by an entropy-penalized maximization over stochastic policies, whose unique solution is the softmax distribution and whose value is the log-sum-exp. Since the log-sum-exp inherits the one-Lipschitz character of the hard maximum, the soft Q-operator is a span-seminorm contraction at the same rate and under the same minorization condition as the classical one, which gives existence, uniqueness up to an additive constant and iterative computability of solutions to the soft optimality equation; the span-seminorm formulation never uses the minorizing measure, which is what makes it compatible with learning.

At the level of the game, the softmax policy depends Lipschitz continuously on the population, with a constant set by the temperature, so that the mean-field operator is a contraction on the simplex whenever the temperature is large enough and $K_1<\varepsilon$, and the equilibrium exists, its population distribution is unique, and it is computed by a geometrically convergent iteration. Replacing the two model-dependent steps of this iteration by synchronous soft relative value $Q$-learning and by an empirical transition kernel yields a model-free algorithm: the reduction of the span-contractive recursion to a supremum-norm contractive stochastic approximation gives an $O(1/n)$ mean-square bound for the inner step through the theory of \citet{ChMaShSh20}, with only a logarithmic dependence on the size of the state-action space, and an inexact Banach iteration argument turns the two errors into explicit convergence rates for the learned population and policy. The numerical experiments on a congestion-aware routing model confirmed these results.

\appendix
\appendixheadings

\section{Proofs for Section~\ref{sec:classical}}\label{app:sec2}
\label{app:classical}

\subsection{Proof of Theorem~\ref{thm:Lhat_span}}
\label{app:thm:Lhat_span}

\begin{proof}{Proof.}
Let $\Phi:=Q-Q'$ and
$\Phi_{\max}(y):=\max_{b\in A(y)}Q(y,b)-\max_{b\in A}Q'(y,b)$. For any $(x,a)$, we have
\[
(LQ)(x,a)-(LQ')(x,a)
=\sum_{y}P(y\mid x,a) \, \Phi_{\max}(y).
\]
If we decompose $P=\widehat{P}+\nu$, we get
\[
\sum_y P(y\mid x,a) \Phi_{\max}(y)
=\sum_y \nu(y)\,\Phi_{\max}(y)
+\sum_y\widehat{P}(y\mid x,a) \Phi_{\max}(y).
\]
The constant $c_\Phi\coloneqq \sum_y \nu(y)\,\Phi_{\max}(y) $ cancels in the span. Since $\widehat{P} \ge 0$ with mass
$1-\varepsilon$,
\[
\spn\!\left(LQ-LQ'\right)
\le(1-\varepsilon)\,\spn(\Phi_{\max}).
\]
If we use $\max_b f-\max_b g \le\max_b (f-g)$ and $\max_b f-\max_b g \geq \min_b (f-g)$, we get
\[
\spn(\Phi_{\max})
\le\max_{y,b}\Phi(y,b)-\min_{y,b}\Phi(y,b)=\spn(\Phi)=\spn(Q-Q').
\]
This completes the proof.\Halmos
\end{proof}

\subsection{Proof of Corollary~\ref{cor:Lhat_banach}}
\label{app:cor:Lhat_banach}

\begin{proof}{Proof.}
Since $L$ is a contraction in the span seminorm,
$\widetilde{L}$ is also a contraction on
$\mathbb{R}^{S\times A}_{(x_0,a_0)}$. Hence, by Banach's fixed point
theorem, it admits a unique fixed point $Q^*$. By definition of $\widetilde{L}$, we have
\[
LQ^*(x,a)
=
LQ^*(x_0,a_0)+Q^*(x,a),
\]
for all $(x,a)\in S\times A$. Setting $\rho^* :=
LQ^*(x_0,a_0)$, we obtain
$\rho^* + Q^*(x,a) = LQ^*(x,a)$,
which is precisely the Q-AROE.

For the second claim, we define
$\Delta_n(x,a):=L^{n+1}Q_0(x,a)-L^n Q_0(x,a)$.
By contraction, we have
\[
\spn(\Delta_n)\le(1-\varepsilon)^n\spn(LQ_0-Q_0)\to 0,
\]
which implies
$\max_{x,a}\Delta_n(x,a)-\min_{x,a}\Delta_n(x,a)\to 0$.
Using the normalization, we get
$\widetilde{L}^n Q_0(x,a)
=
L^n Q_0(x,a)-L^n Q_0(x_0,a_0)$.
Applying $L$ to both sides gives (recall that $L(Q-c) = LQ -c$):
\[
L\widetilde{L}^n Q_0(x,a)
=
L^{n+1}Q_0(x,a)-L^n Q_0(x_0,a_0).
\]
Taking limits and using $\widetilde{L}^n Q_0 \to Q^*$, we obtain when $(x,a) = (x_0,a_0)$,
\[
\lim_{n\to\infty}\Delta_n(x_0,a_0)
=
LQ^*(x_0,a_0)
=
\rho^*.
\]
Finally, since $\spn(\Delta_n)\to 0$, the convergence holds uniformly in
$(x,a)$, which completes the proof.\Halmos
\end{proof}

\section{Proofs for Section~\ref{sec:entropy}}\label{app:sec3}
\label{app:entropy}

\subsection{Proof of Lemma~\ref{lem:lse_lip}}
\label{app:lem:lse_lip}

\begin{proof}{Proof.}
Write $M:=\max_{a\in E}|u(a)-v(a)|$, so $v(a)-M\le u(a)\le v(a)+M$ for every
$a\in E$. Exponentiating and dividing by $\delta$ (monotone operations),
\[
\exp((v(a)-M)/\delta)\le\exp(u(a)/\delta)\le\exp((v(a)+M)/\delta).
\]
Summing over $a\in E$ and using that $\exp(\pm M/\delta)$ pulls out as a constant,
\[
e^{-M/\delta}\sum_{a\in E}\exp(v(a)/\delta)\le\sum_{a\in E}\exp(u(a)/\delta)\le e^{M/\delta}\sum_{a\in E}\exp(v(a)/\delta).
\]
Taking $\delta\log(\cdot)$ throughout (monotone),
\[
-M+\delta\log\sum_{a\in E}\exp(v(a)/\delta)
\le\delta\log\sum_{a\in E}\exp(u(a)/\delta)
\le M+\delta\log\sum_{a\in E}\exp(v(a)/\delta),
\]
which gives the claim.\Halmos
\end{proof}

\subsection{Proof of Theorem~\ref{thm:Lhat_delta_span}}
\label{app:thm:Lhat_delta_span}

\begin{proof}{Proof.}
Let $\Phi:=Q-Q'$ and define
\begin{align*}
\Phi_{\sigma}(y):=\lse_\delta(Q(y,\cdot))-\lse_\delta(Q'(y,\cdot))
=\delta\log\sum_b\exp(Q(y,b)/\delta)
-\delta\log\sum_b\exp(Q'(y,b)/\delta).
\end{align*}
For any $(x,a)$,
\[
(L_\delta Q)(x,a)-(L_\delta Q')(x,a)
=\sum_{y\in S}\Phi_{\sigma}(y)\,P(y\mid x,a).
\]
Decompose $P=\widehat{P}+\nu$ to get
\[
\sum_y\Phi_{\sigma}(y)P(y\mid x,a)
=\sum_y\Phi_{\sigma}(y)\nu(y)
+\sum_y\Phi_{\sigma}(y)\widehat{P}(y\mid x,a).
\]
The constant $c_\Phi \coloneqq \sum_y\Phi_{\sigma}(y)\nu(y) $ does not depend on $(x,a)$ and therefore cancels
in the span. Since $\widehat{P}\ge 0$ with mass
$1-\varepsilon$,
\[
\spn\!\left(L_\delta Q-L_\delta Q'\right)
\le(1-\varepsilon)\,\spn(\Phi_{\sigma}).
\]
Let us bound $\spn(\Phi_{\sigma})$ in terms of $\spn(\Phi)$. For fixed $y$, we have the pointwise inequalities $Q'(y,b)+\min_b\Phi(y,b)\le Q(y,b)\le Q'(y,b)+\max_b\Phi(y,b)$ for all $b\in A(y)$; since the log-sum-exp is monotone and shift-equivariant (Remark~\ref{rem:lse}), applying $\lse_\delta$ to these inequalities gives
\[
\min_b\Phi(y,b)\le\Phi_{\sigma}(y)\le\max_b\Phi(y,b).
\]
Hence
\[
\spn(\Phi_{\sigma})
\le\max_{y,b}\Phi(y,b)-\min_{y,b}\Phi(y,b)
=\spn(\Phi)=\spn(Q-Q').
\]
This completes the proof.\Halmos
\end{proof}

\subsection{Proof of Corollary~\ref{cor:Lhat_soft_banach}}
\label{app:cor:Lhat_soft_banach}

\begin{proof}{Proof.}
Since $L_\delta$ is a contraction in the span seminorm,
the normalized operator
$\widetilde{L}_\delta$
is also a contraction on
$\mathbb{R}^{S\times A}_{(x_0,a_0)}$.
Hence, by Banach's fixed point theorem, there exists a unique
$Q_\delta^*\in\mathbb{R}^{S\times A}_{(x_0,a_0)}$
such that
$\widetilde{L}_\delta Q_\delta^*
=
Q_\delta^*$.
Equivalently,
$L_\delta Q_\delta^*(x,a)
=
L_\delta Q_\delta^*(x_0,a_0)
+
Q_\delta^*(x,a)$
for all $(x,a)\in S\times A$. Therefore,
$\rho_\delta^*
+
Q_\delta^*(x,a)
=
(L_\delta Q_\delta^*)(x,a)$,
where
$\rho_\delta^*
:=
L_\delta Q_\delta^*(x_0,a_0)$.
Thus, the pair $(\rho_\delta^*,Q_\delta^*)$ solves the soft Q-AROE.

For the second claim, define
$\Delta_n(x,a)
:=
L_\delta^{\,n+1}Q_0(x,a)
-
L_\delta^{\,n}Q_0(x,a)$.
Since $L_\delta$ is a span-seminorm contraction,
\[
\spn(\Delta_n)
\le
(1-\varepsilon)^n
\spn(L_\delta Q_0-Q_0)
\to 0.
\]
Hence,
$\max_{x,a}\Delta_n(x,a)
-
\min_{x,a}\Delta_n(x,a)
\to 0$,
which shows that $\Delta_n(x,a)$ becomes asymptotically independent of $(x,a)$.
Next, observe that for every $Q_0$,
$\widetilde{L}_\delta^{\,n}Q_0(x,a)
=
L_\delta^{\,n}Q_0(x,a)
-
L_\delta^{\,n}Q_0(x_0,a_0)$.
Applying $L_\delta$ to both sides and using the translation invariance property
$L_\delta(Q-c)
=
L_\delta Q-c$,
we obtain
\[
L_\delta\widetilde{L}_\delta^{\,n}Q_0(x,a)
=
L_\delta^{\,n+1}Q_0(x,a)
-
L_\delta^{\,n}Q_0(x_0,a_0).
\]
Since
$\widetilde{L}_\delta^{\,n}Q_0\to Q_\delta^*$,
continuity of $L_\delta$ implies
$L_\delta\widetilde{L}_\delta^{\,n}Q_0(x,a)
\to
L_\delta Q_\delta^*(x,a)$.
Choosing $(x,a)=(x_0,a_0)$, we obtain
\[
\lim_{n\to\infty}
\left(
L_\delta^{\,n+1}Q_0(x_0,a_0)
-
L_\delta^{\,n}Q_0(x_0,a_0)
\right)
=
L_\delta Q_\delta^*(x_0,a_0)
=
\rho_\delta^*.
\]
Finally, since
$\max_{x,a}\Delta_n(x,a)
-
\min_{x,a}\Delta_n(x,a)
\to 0$,
the same limit holds for every $(x,a)\in S\times A$. This completes the proof.\Halmos
\end{proof}

\subsection{Proof of Lemma~\ref{lem:hessian_lse}}
\label{app:lem:hessian_lse}

\begin{proof}{Proof.}
Set $p(a):=\sigma_\delta(z)(a)$ and
$Z(z):=\sum_{b=1}^n\exp(z(b)/\delta)$, so
$\lse_\delta(z)=\delta\log Z(z)$ and $p(a)=\exp(z(a)/\delta)/Z(z)$.
For each $a\in\{1,\ldots,n\}$,
\[
\frac{\partial \lse_\delta}{\partial z(a)}
=\frac{\delta}{Z(z)}\cdot\frac{\partial Z}{\partial z(a)}
=\frac{\delta}{Z(z)}\cdot\frac{1}{\delta}\exp(z(a)/\delta)
=\frac{\exp(z(a)/\delta)}{Z(z)}
=p(a),
\]
so $\nabla\lse_\delta(z)=\sigma_\delta(z)$ as claimed.

Let us first compute the diagonal entries of the Hessian. For each $a$,
$\partial^2\lse_\delta/\partial z(a)^2
=\partial p(a)/\partial z(a)$.
Using the quotient rule,
\begin{align*}
\frac{\partial p(a)}{\partial z(a)}
&=\frac{\frac{1}{\delta}\exp(z(a)/\delta)\cdot Z(z)
        -\exp(z(a)/\delta)\cdot\frac{1}{\delta}\exp(z(a)/\delta)}{Z(z)^2}\\
&=\frac{1}{\delta}\cdot\frac{\exp(z(a)/\delta)}{Z(z)}
  -\frac{1}{\delta}\cdot\!\left(\frac{\exp(z(a)/\delta)}{Z(z)}\right)^2
=\frac{1}{\delta}\bigl(p(a)-p(a)^2\bigr).
\end{align*}
We now compute the off-diagonal entries. For $a\neq b$,
$\partial^2\lse_\delta/\partial z(a)\,\partial z(b)
=\partial p(a)/\partial z(b)$.
Now $p(a)=\exp(z(a)/\delta)/Z(z)$ depends on $z(b)$ only through
$Z(z)$, so
\[
\frac{\partial p(a)}{\partial z(b)}
=-\frac{\exp(z(a)/\delta)}{Z(z)^2}\cdot\frac{\partial Z}{\partial z(b)}
=-\frac{\exp(z(a)/\delta)}{Z(z)^2}\cdot\frac{1}{\delta}\exp(z(b)/\delta)
=-\frac{1}{\delta}p(a)p(b).
\]
Therefore, the diagonal entries are $\frac{1}{\delta}(p(a)-p(a)^2)$ and the
off-diagonal entries are $-\frac{1}{\delta}p(a)p(b)$. Together,
\[
\bigl[\nabla^2\lse_\delta(z)\bigr]_{a,b}
=\frac{1}{\delta}\bigl(p(a)\delta_{ab}-p(a)p(b)\bigr),
\]
which in matrix form is exactly~\eqref{eq:hessian_form}.\Halmos
\end{proof}

\subsection{Proof of Lemma~\ref{lem:opnorm_M_linf_l1}}
\label{app:lem:opnorm_M_linf_l1}

\begin{proof}{Proof.}
Fix $v\in\mathbb{R}^n$ with $\|v\|_\infty\le 1$. Since
$(M(p)v)(a)
=
p(a)v(a)-p(a)\sum_b p(b)v(b)$,
we may write
$(M(p)v)(a)
=
p(a)\big(v(a)-\sum_b p(b)v(b)\big)$.
Define
$m:=\sum_b p(b)v(b)$.
Then
$\|M(p)v\|_1
=
\sum_a p(a)\,|v(a)-m|$.
Since the map
$v\mapsto \sum_a p(a)|v(a)-m|$
is convex in $v$, its maximum over the cube
$\{v:\|v\|_\infty\le 1\}$ is attained at an extreme point. Therefore,
it suffices to consider vectors satisfying
$v(a)\in\{-1,1\}$.
Let
$A:=\{a:v(a)=1\}$ and
$\alpha:=\sum_{a\in A}p(a)$.
Then we have
$m
=
\sum_a p(a)v(a)
=
\alpha-(1-\alpha)
=
2\alpha-1$.
Now let us compute:
\[
\|M(p)v\|_1
=
\sum_{a\in A} p(a)|1-m|
+
\sum_{a\notin A} p(a)|-1-m|.
\]
Since $m=2\alpha-1$,
$1-m=2(1-\alpha)$ and
$|-1-m|=2\alpha$.
Therefore,
\[
\|M(p)v\|_1
=
2\alpha(1-\alpha)
+
2(1-\alpha)\alpha
=
4\alpha(1-\alpha).
\]
Finally,
$4\alpha(1-\alpha)\le 1$,
because the quadratic term $\alpha(1-\alpha)$ is maximized at
$\alpha=1/2$, where it equals $1/4$. Hence,
$\|M(p)v\|_1\le 1$.
Taking the supremum over all $\|v\|_\infty\le 1$ completes the proof.\Halmos
\end{proof}

\subsection{Proof of Corollary~\ref{cor:hessian_opnorm}}
\label{app:cor:hessian_opnorm}

\begin{proof}{Proof.}
By Lemma~\ref{lem:hessian_lse},
$\nabla^2\lse_\delta(z)
=
\frac{1}{\delta}M(\sigma_\delta(z))$.
Since $\sigma_\delta(z)\in\Delta^{n-1}$,
Lemma~\ref{lem:opnorm_M_linf_l1} yields
$\|M(\sigma_\delta(z))\|_{\infty\to 1}\le 1$.
Therefore,
$\bigl\|\nabla^2\lse_\delta(z)\bigr\|_{\infty\to 1}
=
\frac{1}{\delta}
\bigl\|M(\sigma_\delta(z))\bigr\|_{\infty\to 1}
\le
\frac{1}{\delta}$.\Halmos
\end{proof}

\subsection{Proof of Lemma~\ref{lem:softmax_linf_l1}}
\label{app:lem:softmax_linf_l1}

\begin{proof}{Proof.}
Since $\sigma_\delta=\nabla\lse_\delta$ is the gradient of a
$C^2$ function, the fundamental theorem of calculus along the line
segment joining $z'$ and $z$ yields
\[
\sigma_\delta(z)-\sigma_\delta(z')
=
\int_0^1
\nabla^2\lse_\delta\bigl(z'+t(z-z')\bigr)(z-z')
\,\mathrm{d}t.
\]
Taking the $\|\cdot\|_1$-norm and using the triangle inequality together
with Corollary~\ref{cor:hessian_opnorm}, we obtain
\begin{align*}
\|\sigma_\delta(z)-\sigma_\delta(z')\|_1
&\le
\int_0^1
\bigl\|
\nabla^2\lse_\delta(z'+t(z-z'))(z-z')
\bigr\|_1
\,\mathrm{d}t \\
&\le
\int_0^1
\bigl\|
\nabla^2\lse_\delta(z'+t(z-z'))
\bigr\|_{\infty\to 1}
\,\|z-z'\|_\infty
\,\mathrm{d}t \\
&\le
\int_0^1
\frac{1}{\delta}\|z-z'\|_\infty
\,\mathrm{d}t
=
\frac{1}{\delta}\|z-z'\|_\infty.
\end{align*}
For the second claim, the shift-invariance of the softmax follows directly from the defining formula, so that $\|\sigma_\delta(z)-\sigma_\delta(z')\|_1=\|\sigma_\delta(z)-\sigma_\delta(z'+c\1)\|_1\le\frac1\delta\|z-z'-c\1\|_\infty$ for every $c$; minimizing over $c$ (the optimal shift is
$c^*=\tfrac12(\max_a (z-z')(a)+\min_a (z-z')(a))$) gives~\eqref{eq:softmax_span_lip}.\Halmos
\end{proof}

\bibliographystyle{informs2014}
\bibliography{paper-mr}

\end{document}